\documentclass[11pt,reqno]{amsart}
\usepackage{amsmath,amssymb,graphicx,mathrsfs,amsopn,stmaryrd,amsthm,color,bm,extarrows}
\usepackage{newtxtext}
\usepackage[dvipsnames]{xcolor}
\usepackage[colorlinks=true,citecolor=blue,linkcolor=blue
]{hyperref}
\usepackage[final]{showlabels}
\usepackage[english]{babel}
\usepackage[mathscr]{euscript}
\usepackage[alphabetic,nobysame]{amsrefs}
\renewcommand{\MR}[1]{} \renewcommand{\PrintDOI}[1]{}

\newcommand{\arxiv}[1]{\href{http://arxiv.org/abs/#1}{\sf arXiv:\nolinkurl{#1}}}
\usepackage{geometry}
\usepackage{dynkin-diagrams}[root radius = 2.5cm]
\usepackage{collectbox}
\DeclareFontFamily{OT1}{pzc}{}
\DeclareFontShape{OT1}{pzc}{m}{it}{<-> s * [1.10] pzcmi7t}{}
\DeclareMathAlphabet{\mathpzc}{OT1}{pzc}{m}{it}
\makeatletter 
\newcommand{\mybox}{%
    \collectbox{%
        \setlength{\fboxsep}{1pt}%
        \fbox{\BOXCONTENT}%
    }%
}
\makeatother

\usepackage{dsfont}
\usepackage{tikz}    
\usetikzlibrary{arrows,matrix,decorations.markings,shapes.geometric}   
\usetikzlibrary{decorations.pathreplacing}

\allowdisplaybreaks

\numberwithin{equation}{section}
\newtheorem{thm}{Theorem}[section]
\newtheorem{prop}[thm]{Proposition}
\newtheorem{lem}[thm]{Lemma}
\newtheorem{cor}[thm]{Corollary}

\theoremstyle{definition} 
\newtheorem{eg}[thm]{Example}
\newtheorem{dfn}[thm]{Definition}
\theoremstyle{remark}
\newtheorem{rem}[thm]{Remark}

\newcommand{\beq}{\begin{equation}}
\newcommand{\eeq}{\end{equation}}
\newcommand{\be}{\begin{equation*}}
\newcommand{\ee}{\end{equation*}}

\newcommand{\bC}{\mathbb{C}}

\newcommand{\bZ}{\mathbb{Z}}

\newcommand{\bN}{\mathbb{N}}

\newcommand{\mc}{\mathcal}

\newcommand{\cR}{\mathcal{R}}

\newcommand{\gl}{\mathfrak{gl}}
\newcommand{\fkn}{\mathfrak{n}}

\newcommand{\fkm}{\mathfrak{m}}

\newcommand{\fks}{\mathfrak{s}}

\newcommand{\rY}{\mathrm{Y}}

\newcommand{\End}{\mathrm{End}}

\newcommand{\gr}{{\mathrm{gr}}}

\newcommand{\pa}{\partial}
\newcommand{\tl}{\tilde}
\newcommand{\wtl}{\widetilde}
\newcommand{\gge}{\geqslant}
\newcommand{\lle}{\leqslant}
    
\newcommand{\ka}{{\mathfrak s}}

\newcommand{\bs}{{\bm\fks}}

\newcommand{\Sym}{{\mathrm{Sym}}}
\newcommand{\Y}{{\mathscr{Y}}}
\newcommand{\scrX}{{\mathscr{X}}}

\newcommand{\SY}{{\mathscr{SY}}}

\newcommand{\rU}{{\mathrm{U}}}
\newcommand{\Yi}{{\mathbf{Y}^\bs_\imath}}

\def \h{\mathfrak{h}}
\def \H{h}
\def \X{b}

\def \I{\mathbb{I}}
\def \N{\mathbb{N}}
\def \Sym{\mathrm{Sym}}

\def \Z{\mathbb{Z}}

\def \C{\mathbb{C}}

\begin{document}
\pagestyle{myheadings}
\setcounter{page}{1}

\title[Twisted super Yangians of quasi-split type A]{Twisted super Yangians of quasi-split type A} 

\date{\today}
\author{Kang Lu}
\address{
Shenzhen International Center for Mathematics and Department of Mathematics, Southern University
of Science and Technology, Shenzhen, China
}\email{kanglu.math@outlook.com,luk@sustech.edu.cn}

		
\begin{abstract}
We introduce the twisted super Yangian $\Yi$ of quasi-split type A in the Drinfeld current presentation for an arbitrary symmetric parity sequence $\bm\ka$. We prove via Gauss decomposition that the twisted super Yangian $\Yi$ is isomorphic to the (special) twisted super Yangian $\Y^{\bm\ka}$, previously defined via the R-matrix presentation. As a corollary, we prove that the twisted super Yangians $\Yi$ corresponding to different parity sequences with the same $\mathfrak m$ and $\mathfrak n$ are isomorphic. Additionally, we establish the PBW theorem for $\Yi$ and describe the center of $\Y^{\bm\ka}$ in terms of Gaussian generators, thereby generalizing known results for the nonsuper quasi-split type A case.
\end{abstract}
	
\maketitle
\setcounter{tocdepth}{1}
\tableofcontents

\thispagestyle{empty}
\section{Introduction}
Inspired by Cherednik’s scattering theory for factorized particles on the half-line \cite{Cherednik1984Factorizing}, Sklyanin \cite{Sklyanin1988Boundary} introduced a class of algebras defined by reflection equations in the FRT formalism \cite{FRT89}, leading to the construction of quantum integrable systems with boundary conditions via the quantum inverse scattering (R-matrix) method. The reflection equations are central to constructing the commutative Bethe subalgebra, ensuring the integrability of the associated integrable systems with boundary conditions.

Twisted Yangians are coideal subalgebras of Yangians associated to symmetric pairs. They are in general quotients of reflection algebras by symmetry or unitary relations. A first example of twisted Yangians is due to
Olshanski \cite{Olshanski1992twisted} who constructed twisted Yangians of types AI and AII in R-matrix presentation. These algebras are closely related to representations of classical Lie algebras \cite{Molev2007book}. Recently, Bethe vectors and recurrence relations for open spin chains, whose symmetry is described by Olshanski's twisted Yangians, were uniformly studied in \cite{regelskis2024bethe}; see also references therein. The R-matrix construction of twisted Yangians was later extended to type AIII \cite{Molev2002reflection} and to symmetric pairs of classical types \cite{Guay2016twisted}. Additionally, a construction of twisted Yangians for general symmetric pairs via Drinfeld's J-presentation has emerged as boundary remnants of Yangians in 1+1D integrable field theories \cite{MacKay2002rational,Belliard2017}.

In this article, we continue our study for a special case of reflection superalgebras, referred to as twisted super Yangians of quasi-split type A; cf.~\cite{lu2023twisted,lu2024drinfeld}. These twisted super Yangians are super analogues of the reflection algebra introduced in \cite{Molev2002reflection}; cf. also \cite{Chen2014twisted}. They are associated with supersymmetric pairs of type AIII, while the twisted super Yangians of types AI and AII were previously introduced in \cite{Briot2003twisted} and recently studied in \cite{Lin2024from,Lin2025Ber}. 

The twisted super Yangians of type AIII (also known as \textit{reflection superalgebras}) have appeared in the study of open spin chains with diagonal boundary conditions, employing both analytic and algebraic Bethe ansatz approaches \cite{Ragoucy2007analytical,Belliard2009nested}. Moreover, the (super)trace formula \cite{Tarasov13comb} of Bethe vectors and the Bethe ansatz equations were obtained in \cite{Belliard2009nested}. Twisted super Yangians in standard parity sequence (with specific diagonal boundary matrices) were further investigated in \cite{kettle2023orthosymplectic,bagnoli2023double} where some basic properties of the superalgebras were established. In our prior work \cite{lu2023twisted}, we worked on twisted super Yangians in R-matrix presentation associated to arbitrary parity sequences and arbitrary diagonal boundary matrices, established a highest weight representation theory in this general setting, classified finite-dimensional irreducible representations for certain cases (cf. \cite{Molev2002reflection}), and extended the Schur-Weyl type duality between degenerate affine Hecke algebras of type BC and twisted super Yangians (cf. \cite{Chen2014twisted}).

The goal of this paper is to introduce a Drinfeld type presentation for twisted super Yangians of quasi-split type A, associated with arbitrary symmetric parity sequences, and to generalize the joint work with Weinan Zhang \cite{lu2024drinfeld} to the super case. We establish an explicit isomorphism between twisted super Yangians in R-matrix and Drinfeld presentations via Gauss decomposition; cf. \cite{Brundan2005parabolic,Jing18iso,frassek2023orthosymplectic,Lu2023drinfeld,Molev24drinfeld}. Recently, the Gauss decomposition and Drinfeld presentation for twisted Yangians are used to investigate the classical Slodowy slices \cite{Tappeiner2024shifted}, finite W-algebras \cite{Lu2025shifted-W}, and fixed point loci of affine Grassmannian slices \cite{Lu2025shifted-G}.

A new feature of Lie superalgebras is the existence of multiple nonisomorphic Dynkin diagrams. In type A, these Dynkin diagram can be effectively described by a parity sequence $\bs$. We then consider the general linear Lie superalgebra $\gl_{\mathfrak m|\mathfrak n}^\bs$ where $\bs=(\ka_1,\dots,\ka_{\fkm+\fkn})$ is a parity sequence such that $\ka_i=\pm 1$ and the occurrence of $1$ is exactly $\mathfrak m$. 

A supersymmetric pair $(\gl_{\mathfrak m|\mathfrak n}^\bs, \mathfrak k_{\fkm|\fkn}^\bs)$ is of quasi-split type if $\mathfrak k_{\fkm|\fkn}^\bs$ is the fixed point subalgebra of an involution $\theta=\omega\circ \tau$, where $\tau$ is a nontrivial involution of the underlying Dynkin diagram and $\omega$ is the Cartan involution. In the type A case, a nontrivial $\tau$ is uniquely given by $\tau i=\fkm+\fkn-i$, enforcing symmetry in the Dynkin diagram of $\gl_{\mathfrak m|\mathfrak n}^\bs$, meaning that nodes $i$ and $\fkm+\fkn-i$ must have the same parity. In the present paper, we consider \emph{all} such Dynkin diagrams with one natural restriction: if $\tau i=i$ (which occurs only if $\fkm+\fkn$ is even), then the node $i$ has to be even; see \cite[Def.~2.3~\&~Ex.~4.9]{ShenWang2024quantum}. All such Dynkin diagrams are described by the \textit{symmetric} parity sequences. Here a parity sequence $\bs$ is symmetric if  $\ka_i=\ka_{\fkm+\fkn+1-i}$.

For a symmetric parity sequence $\bs$, we introduce a  superalgebra $\Yi$, referred to as a twisted super Yangian in Drinfeld presentation (see Definition \ref{deftY}). These superalgebras are super analogues of twisted Yangians introduced in \cite{lu2024drinfeld}. We then establish an explicit isomorphism between the new superalgebras and the twisted super Yangians $\Y^\bs$ in R-matrix presentation. Our approach utilizes the well-studied Gauss decomposition method, extensively employed to provide explicit isomorphisms between Yangians and quantum affine algebras of classical types in R-matrix and Drinfeld presentations, including the supersymmetric setting \cite{Brundan2005parabolic,Gow2007gauss,Peng2016parabolic,Jing18iso,frassek2023orthosymplectic,Molev24drinfeld}. 

In contrast to the type AI case, where the relations between Gaussian generators are twisted analogues of \cite{Brundan2005parabolic}, the quasi-split type A (or quasi-split type AIII) closely resembles the classical type BCD cases investigated in \cite{Jing18iso,frassek2023orthosymplectic,Molev24drinfeld}. By carefully modifying the R-matrix presentation, the super Yangian of type A can be naturally regarded as a subalgebra of twisted super Yangians by focusing on the upper left half of the generating matrix of the twisted super Yangians $\Y^\bs$. Consequently, many relations can be directly adapted from \cite{Gow2007gauss,Peng2016parabolic,tsymbaliuk2020shuffle}. However, a nontrivial rank reduction homomorphism is crucial for effectively reducing the calculation of new relations to small rank cases. We establish the rank reduction homomorphism (actually embedding) in Proposition \ref{embedB} by employing techniques from \cite{Jing18iso}; see also \cite[Prop. 6.3]{lu2024drinfeld}. Unlike the orthosymplectic Yangians, this technique works well in the super setting as $R(u)$ (Yang's rational R-matrix) does not have singularities at $\ka_1$ (cf. \cite{frassek2023orthosymplectic,Molev24drinfeld}). Consequently, we achieve our main result, an explicit isomorphism between twisted super Yangians in Drinfeld and R-matrix presentations; see Theorem \ref{main2}. Meanwhile, we obtain a PBW theorem for the twisted super Yangians in current generators for both presentations; see Corollary \ref{thm:pbw} and Theorem \ref{main1}. 

The twisted super Yangians in R-matrix presentation possess nontrivial centers, and constructing these centers (Sklyanin superdeterminant) in terms of R-matrix generators remains open due to the absence of one-dimensional modules in tensor products of natural representations. Nevertheless, we construct a central series $\mathpzc{Ber}^\bs(u)$ (termed quantum Berezinian, cf. \cite{Nazarov1991berezinian}) using the Cartan type generating series obtained from the Gauss decomposition in \eqref{cu} and prove that its coefficients of $u^{-2r-1}$ ($r\in\bN$) are free generators of the center of the twisted super Yangian, generalizing \cite[Thm.~6.19]{lu2024drinfeld}. Interestingly, the shifts and formulas for the central series mirror those in the nontwisted case; see \cite{Gow2007gauss,tsymbaliuk2020shuffle,Huang20duality,Chang23center}.

The twisted super Yangians $\Yi$ associated with different parity sequences $\bs$ that share the same $\fkm$ and $\fkn$ are isomorphic. This fact, though nontrivial in terms of Drinfeld presentation, is naturally apparent for $\Y^\bs$ in R-matrix presentation, as the superalgebras are isomorphic via index permutation. As a corollary of our main result, we show that $\Yi$ for different $\bs$ are indeed isomorphic; see Theorem~\ref{main3} and cf. \cite{tsymbaliuk2020shuffle}. Moreover, the central series $\mathpzc{Ber}^\bs(u)$ remains invariant under the permutation isomorphisms of $\Y^\bs$, extending results from \cite{Huang20duality,Chang23center} to the twisted case; see Theorem~\ref{main4}.

The Gauss decomposition approach should also apply to twisted Yangians of classical type introduced in \cite{Guay2016twisted}, although some modifications in the presentation might be necessary. It is natural to anticipate that these results can be extended to their $q$-analogues (cf. \cite{LuWangZ2024braid}), as has been successfully done for the type AI case in \cite{lu2024isomorphism}. Notably, there are no split types in the supersymmetric case, making the quasi-split type the natural first case to study in the supersymmetric setting.

Let us indicate more explicitly the new features of the super case. The general framework of the proof, including Gauss decomposition, rank reduction, and a suitable choice of the R-matrix presentation, is adapted from the nonsuper quasi-split type A case \cite{lu2024drinfeld}. The super setting, however, introduces several additional ingredients. The presence of odd simple roots leads to a quartic Serre relation, which has no counterpart in the ordinary nonsuper case. Moreover, the center is described by a quantum Berezinian constructed from Gaussian generators. This provides a natural replacement, in the present setting, for the Sklyanin determinant description familiar from the nonsuper R-matrix theory.

For the reader's convenience, we summarize the main algebras used in the paper and the maps relating them. The algebra $\mathbf Y^\bs_\imath$ is the Drinfeld-current presentation introduced in Definition~\ref{deftY}. On the R-matrix side, $\mathscr X^\bs$ denotes the extended twisted super Yangian; see Definition~\ref{eradef}. The algebra $\mathscr Y^\bs$ is its quotient by the unitary relation; see Definition~\ref{def:R-Y} and \eqref{unimat}. Finally, the special twisted super Yangian $\mathscr{SY}^\bs$ is the subalgebra of $\mathscr Y^\bs$ defined in Definition~\ref{defSY}. These algebras are related by
\[
\mathbf Y^\bs_\imath \cong \mathscr{SY}^\bs \hookrightarrow \mathscr Y^\bs \twoheadleftarrow \mathscr X^\bs.
\]
One of the main results of this paper is the first isomorphism, which is established via the Gaussian generators of the R-matrix presentation.

The paper is organized as follows. In Section \ref{tYDrinfeld}, we introduce the twisted super Yangians in Drinfeld presentations and study their basic properties. Section \ref{sec:R} is devoted to a review of super Yangians and twisted super Yangians in R-matrix presentations. We investigate the Gauss decomposition of the generating matrix of twisted super Yangians and establish the rank reduction homomorphism in Section \ref{sec:GDmain}. Basic symmetries for the Gaussian generators are also discussed. We formulate and prove our main results in Section \ref{sec:mainres}. In Section \ref{sec:lowrk}, we verify the relations in small ranks which are crucial to the proof of our main results.

\medskip

\noindent {\bf Acknowledgment.} 
The author thanks Yaolong Shen and Weiqiang Wang for stimulating discussions on quantum supersymmetric pairs, and  Weinan Zhang for collaborations on the nonsuper case.

\section{Twisted super Yangians in Drinfeld presentation}
\label{tYDrinfeld}
\subsection{Lie superalgebra}Throughout the paper, we work over $\bC$. In this subsection, we recall the basics of the general linear Lie superalgebra, see e.g. \cite{Cheng2009dualities} for more detail.

A \emph{vector superspace} $W = W_{\bar 0}\oplus W_{\bar 1}$ is a $\Z_2$-graded vector space. We call elements of $W_{\bar 0}$ \emph{even} and elements of $W_{\bar 1}$ \emph{odd}. We write $|w|\in\{\bar 0,\bar 1\}$ for the parity of a homogeneous element $w\in W$. Set $(-1)^{\bar 0}=1$ and $(-1)^{\bar 1}=-1$.

Fix $\fkm,\fkn\in \Z_{\gge 0}$ and denote $N=\fkm+\fkn$. Set $\I^0:=\{1,2,\dots,N-1\}$ and $\I:=\{1,2,\dots,N\}$. For each $i\in\I$, set $i':=N+1-i$. 

Denote by $S_{\fkm|\fkn}$ the set of all sequences $\bs=(\ka_{1},\ka_2,\dots,\ka_{N})$ where $\ka_i\in\{\pm1\}$ and $1$ occurs exactly $\fkm$ times. Elements of $S_{\fkm|\fkn}$ are called \emph{parity sequences}. We call a parity sequence \emph{symmetric} if $\ka_i=\ka_{i'}$ for all $i\in\I$. There exists at least one symmetric parity sequence in $S_{\fkm|\fkn}$ if and only if $\mathfrak m\mathfrak n$ is even.

Fix a parity sequence $\bs\in S_{\fkm|\fkn}$ and define $|i|\in \Z_2$ for $i\in \I$ by $\ka_i=(-1)^{|i|}$.

It is well known that there are different nonconjugate  root systems of the general linear Lie superalgebra $\gl_{\fkm|\fkn}$ and they are parameterized by parity sequence in $S_{\fkm|\fkn}$.

The Lie superalgebra $\gl_{\fkm|\fkn}^\bs$ is generated by elements $e_{ij}$, $i,j\in \I$, with the supercommutator relations
\[
[e_{ij},e_{kl}]=\delta_{jk}e_{il}-(-1)^{(|i|+|j|)(|k|+|l|)}\delta_{il}e_{kj},
\]
where the parity of $e_{ij}$ is $|i|+|j|$. Since $\fkm$ and $\fkn$ can be determined from $\bs$, we simply write $\gl_{\fkm|\fkn}^\bs$ as $\gl^\bs$. The Lie superalgebra $\mathfrak{sl}^\bs$ is the Lie subalgebra of $\gl^\bs$ spanned by the vectors $e_{ij}$ and $\ka_ie_{ii}-\ka_je_{jj}$ for $i,j\in\I$ such that $i\ne j$.

The \textit{Cartan subalgebra $\h$} of $\gl^\bs$ is spanned by $e_{ii}$, $i \in\I$. Let $\epsilon_i$, $i \in\I$, be a basis of $\h^*$ (the dual space of $\h$) such that $\epsilon_i(e_{jj})=\delta_{ij}$. There is a bilinear form $(\cdot,\cdot)$ on $\h^*$ given by $(\epsilon_i,\epsilon_j)=\ka_i\delta_{ij}$. The \textit{root system $\bf{\Phi}$} is a subset of $\h^*$ given by
\[
{\bf \Phi}:=\{\epsilon_i-\epsilon_j~\vert ~i,j\in\I \text{ and }i\ne j\}.
\]
Let $\mc R^+=\{\epsilon_i-\epsilon_j~|~1\lle i<j\lle N\}$ be the set of positive roots. We call a root $\epsilon_i-\epsilon_j$ \textit{even} (resp. \textit{odd}) if $\vert i\vert =\vert j\vert $ (resp. $\vert i\vert \ne \vert j\vert $). Set $\alpha_i:=\epsilon_i-\epsilon_{i+1}$ for $i\in\I^0$. Then we have $|\alpha_i|=|i|+|i+1|$.

The \textit{symmetric Cartan matrix} $C^\bs=(c_{ij})_{i,j\in\I^0}$ associated to the parity sequence $\bs$ is given by
\[
c_{ij}=(\alpha_i,\alpha_j).
\]
For example, $c_{ii}=\ka_i+\ka_{i+1}$ and $c_{i,i+1}=c_{i+1,i}=-\ka_{i+1}$.

For a symmetric parity sequence, the associated Dynkin diagram is also symmetric. We shall give a few examples. Here we use $\dynkin[root radius = 0.1cm] A{o}$ for even roots and $\dynkin[root radius = 0.1cm] A{t}$ for odd roots.
For the parity sequences $(-1,-1,1,-1,-1)$ and $(-1,-1,1,1,-1,-1)$, the corresponding Dynkin diagrams, respectively, are given by
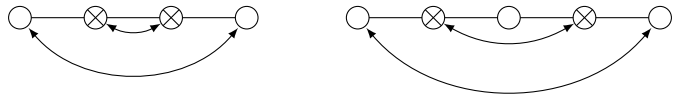
\[
\dynkin[root radius = 0.15cm,edge length=1cm,involutions={[in=-130,out=-50,relative]14;23}] A{otto} \quad\qquad \dynkin[root radius = 0.15cm,edge length=1cm,involutions={[in=-130,out=-50,relative]15;24;33}] A{ototo}
\]
Note that if $N$ is even (i.e. the rank is odd), then the node in the middle is always an even root. We rephrase these more precisely.

Let $\bs$ be a symmetric parity sequence. Let $\tau:\I^0\to \I^0$ be the bijection, $\tau i=N-i$. Then $\tau$ is a Dynkin diagram involution, i.e. $c_{ij}=c_{\tau i,\tau j}$. If $N=2\ell$ is even, then $\ell$ is a fixed point of $\tau$, i.e. $\tau \ell=\ell$. In this case, we always have $|\alpha_\ell|=\bar 0$. 

Throughout the paper, we use two closely related involutions on different index sets: for \(i\in \mathbb I=\{1,\ldots,N\}\), we write \(i'=N+1-i\), while for \(i\in \mathbb I^0=\{1,\ldots,N-1\}\), the Dynkin diagram involution is denoted by \(\tau i=N-i\). Thus \(\tau i=(i+1)'=i'-1\) for \(i\in\mathbb I^0\).

Let $\theta$ be the involution of $\gl^\bs$ defined by
\[
\theta:\gl^\bs\to \gl^\bs,\quad e_{ij}\mapsto (-1)^{i-j}e_{i'j'}.
\]
The involution $\theta$ can be thought of as the composition of the involution $\tau$ and the Chevalley involution. Let $\mathfrak k^\bs$ be the fixed point subalgebra of $\gl^\bs$ under the involution $\theta$. Then the pair $(\gl^\bs,\mathfrak k^\bs)$ is called a \textit{supersymmetric pair of quasi-split type A} (or \textit{quasi-split type AIII}, to be more precisely), cf. \cite{Kolb2020nichols,Shen2024quantum,algethami2024quantum,ShenWang2024quantum}. 

\subsection{Drinfeld presentation}
From now on, we fix an arbitrary symmetric parity sequence $\bs\in S_{\fkm|\fkn}$. For homogeneous elements $x,y$ in a superalgebra, we write
\[
[x,y]=xy-(-1)^{|x||y|}yx,\qquad \{x,y\}=xy+(-1)^{|x||y|}yx.
\]

The following presentation is motivated by the Gauss decomposition of the generating matrix of the twisted super Yangian in the R-matrix presentation. More precisely, the generators \(h_{i,r}\) and \(b_{i,r}\) should be viewed as abstract counterparts of the Cartan and root-type Gaussian generators introduced later in \eqref{beven}--\eqref{hodd}. The defining relations below are chosen so that these Gaussian generators satisfy them; this will be verified in Sections \ref{sec:GDmain} and \ref{sec:lowrk}, and will ultimately yield the isomorphism theorem in Theorem \ref{main2}.

\begin{dfn}\label{deftY}
The \emph{twisted super Yangian of quasi-split type A} (in Drinfeld presentation) associated with the symmetric parity sequence $\bm\ka$, denoted by $\Yi$ (the subscript $\imath$ is borrowed from the theory of $\imath$-quantum groups, i.e. quantum symmetric pair coideal subalgebras, and $\imath$ stands for involution), is the unital superalgebra generated by $\H_{i,r}$, $\X_{i,r}$, $i\in \mathbb{I}^0$, $r\in \N$, where $|h_{i,r}|=\bar 0$ is even and $b_{i,r}$ is of parity $|\alpha_i|$, subject to
\begin{align}\label{qsconj0}
&\quad [\H_{i,r},\H_{j,s}]=0,\qquad\qquad  h_{\tau i,0}=-h_{i,0},
\\\notag
&\quad [\H_{i,r},\X_{j,s}] - [\H_{i,r-2},\X_{j,s+2}]
\\\label{qsconj2}
&\hskip 0.7cm =\frac{c_{ij}-c_{\tau i,j}}{2} \{\H_{i,r-1},\X_{j,s}\}+\frac{c_{ij}+c_{\tau i,j}}{2} \{\H_{i,r-2},\X_{j,s+1}\}+\frac{c_{ij}c_{\tau i,j}}{4} [\H_{i,r-2}, \X_{j,s}],
\\\label{qsconj3}
&\quad [\X_{i,r+1 },\X_{j,s}]  - [\X_{i,r },\X_{j,s+1 }] 
 =\frac{c_{ij}}{2} \{\X_{i,r },\X_{j,s }\}-2 \delta_{\tau i,j}(-1)^{r}\ka_i \H_{j,r+s+1 },
\end{align}
and the Serre relations: for $c_{ij}=0,$
\begin{align}\label{qsconj4}
&[\X_{i,r},\X_{j,s}]= \delta_{\tau i,j}(-1)^{r}\ka_i h_{j,r+s},
\end{align}
and for $j\neq \tau i\neq i,j=i\pm 1,|\alpha_i|=\bar 0$,
\begin{align}\label{qsconj9}
&\mathrm{Sym}_{k_1,k_2}\big[b_{i,k_1},[b_{i,k_2},b_{j,r}] \big] =0,
\end{align}
and for $|\alpha_{i-1}|=|\alpha_{i+1}|=\bar 0$, $|\alpha_i|=\bar 1$,
\beq\label{qsconjnew}
\big[[b_{i-1,r},b_{i,0}],[b_{i,0},b_{i+1,s}]\big]=0,
\eeq
and for $N=2\ell,j=\ell\pm 1$,
\beq\label{qsconj8}
\begin{split}
 \mathrm{Sym}_{k_1,k_2}&\big[b_{\ell,k_1},[b_{\ell,k_2},b_{j,r}] \big] \\
=&\,(-1)^{k_1}\sum_{p\gge 0}2^{-2p} \big(\ka_\ell [h_{\ell,k_1+k_2-2p-1},b_{j,r+1}]-\{h_{\ell,k_1+k_2-2p-1},b_{j,r}\}\big),
\end{split}
\eeq
and for $N=2\ell+1,i\in\{\ell,\ell+1\}, |\alpha_\ell|=\bar 0$,
\begin{align}\label{qsconj10}
&\mathrm{Sym}_{k_1,k_2}\big[b_{i,k_1},[b_{i,k_2},b_{\tau i,r}] \big] =\frac{4}{3}\,\mathrm{Sym}_{k_1,k_2}(-1)^{k_1}\sum_{p=0}^{k_1+r}3^{-p}\ka_i[b_{i,k_2+p},h_{\tau i,k_1+r-p}],
\end{align}
where $\H_{i,s}=0$ if $s< -1$ and $h_{i,-1}=1$.
\end{dfn}

\begin{lem}\label{lem:sym-new}
We have $h_{\tau i,r}=(-1)^{r+1}h_{i,r}$ for all $i\in \I^0$ and $r\gge 0$.
\end{lem}
\begin{proof}
By \eqref{qsconj0}, it suffices to prove it for the case $r>0$. By \eqref{qsconj3}, we have
\begin{align*}
[\X_{i,r+1 },\X_{\tau i,s}]  - [\X_{i,r },\X_{\tau i,s+1 }] 
 =\frac{c_{i\tau i}}{2} \{\X_{i,r },\X_{\tau i,s }\}-2 (-1)^{r}\ka_i \H_{\tau i,r+s+1 },\\
[\X_{\tau i,s+1 },\X_{i,r }]   - [\X_{\tau i,s},\X_{i,r+1 }]
 =\frac{c_{\tau i, i}}{2} \{\X_{\tau i,s },\X_{i,r }\}-2 (-1)^{s}\ka_{\tau i} \H_{i,r+s+1 }.
\end{align*}
Note that $|\alpha_i|=|\alpha_{\tau i}|$ and $(-1)^{|\alpha_i|}\fks_i=\fks_{i+1}=\fks_{N-i}=\fks_{\tau i}$. Therefore, the desired equality for $r>0$ follows by comparing the two displayed formulas.
\end{proof}

\begin{rem}
By setting $r=0$ and $r=1$ in \eqref{qsconj2}, we have
\begin{align}
\label{qsconj5'}
&\quad [\H_{i,0},\X_{j,r}]=  (c_{ij}-c_{\tau i,j}) \X_{j,r},
\\\label{qsconj5}
&\quad [\H_{i,1},\X_{j,r}]=  (c_{ij}+c_{\tau i,j}) \X_{j,r+1}+\frac{ c_{ij}-c_{\tau i,j}}{2}\{\H_{i,0},\X_{j,r}\}.    
\end{align}
\end{rem}

If $c_{\tau i,j}=0$, then the relation \eqref{qsconj2} has a more familiar equivalent form, which corresponds to the current relations for the ordinary super Yangians; see also \S\ref{sec:hb} and \S\ref{secA}. 
\begin{lem}\label{lemalt}
If $c_{\tau i,j}=0$, then the relation \eqref{qsconj2} is equivalent to 
\beq\label{alt}
[h_{i,r+1},b_{j,s} ]-[h_{i,r},b_{j,s+1} ]= \frac{c_{ij}}{2} \{h_{i,r},b_{j,s}\}.
\eeq    
\end{lem}
\begin{proof}
The proof is the same as that of \cite[Lemma 3.3]{lu2024drinfeld}.
\end{proof}

\subsection{A PBW spanning set}
For $\alpha\in\mc R^+$ and $r\in\bN$, where $\alpha=\alpha_i+\alpha_{i+1}+\cdots+\alpha_{j-1}$ for some $i,j$ such that $1\lle i<j\lle N$, define
\beq\label{rv}
b_{\alpha,r}:=\Big[b_{j-1,0},\big[b_{j-2,0},\cdots[b_{i+1,0},b_{i,r}]\cdots\big]\Big].
\eeq
Set
\beq\label{Itau}
\begin{split}
&\I_{\ne}^0=\{1,\dots,\ell-1\},\qquad \I_{=}^0=\{\ell\},\qquad \text{ if }N=2\ell;\\
&\I_{\ne}^0=\{1,\dots,\ell\},\hskip 1.45cm \I_{=}^0=\varnothing,\hskip 1.cm \text{ if }N=2\ell+1.
\end{split}
\eeq
An \textit{ordered super monomial} is an ordered monomial containing
no second or higher order powers of the odd generators.
\begin{prop}\label{prop:span}
The ordered super monomials of
\beq\label{eq:span}
\big\{b_{\alpha,r},h_{i,r},h_{j,2r+1}|\alpha\in\mc R^+,i\in \I^0_{\ne}, j\in \I^0_{=},r\in\bN\big\}
\eeq
(with respect to any fixed total ordering) form a spanning set of $\Yi$.
\end{prop}
\begin{proof}
Define a filtration on the superalgebra $\Yi$ by setting $\deg b_{i,r}=\deg h_{i,r}=r+1$ and denote by $\tl\gr\Yi$ the associated graded superalgebra. Denote by $\tl h_{i,r},\tl b_{i,r}$ the images of $h_{i,r},b_{i,r}$ in the $(r+1)$-st component in $\tl\gr\Yi$. Then the relations \eqref{qsconj2}, \eqref{qsconj5'}, and \eqref{qsconj5} imply that $[\tl h_{i,r},\tl b_{j,s}]=0$ for $i,j\in\I^0$ and $r,s\in\bN$. By \eqref{qsconj0}, $\tl h_{i,r}$ and $\tl h_{j,s}$ commute in $\tl\gr\Yi$. Moreover, the relations \eqref{qsconj3}--\eqref{qsconj10} imply that
\beq\label{tobea}
\begin{split}
[\tl b_{i,r+1},\tl b_{j,s}]-[\tl b_{i,r},\tl b_{j,s+1}]=0,&\\
[\tl b_{i,r},\tl b_{j,s}]=0,&\qquad \text{ if }c_{ij}=0,\\
\Sym_{k_1,k_2}\big[\tl b_{i,k_1},[\tl b_{i,k_2},\tl b_{j,r}]\big]=0,& \qquad \text{ if }|\alpha_i|=\bar 0 \text{ and }j=i\pm 1,\\
\big[[\tl b_{i-1,r},\tl b_{i,0}],[\tl b_{i,0},\tl b_{i+1,s}]\big]=0,&\qquad \text{ if }|\alpha_i|=\bar 1 \text{ and }|\alpha_{i-1}|=|\alpha_{i+1}|=\bar 0.
\end{split}
\eeq
Together with Lemma \ref{lem:sym-new}, we conclude that $\tl\gr\Yi$ is a quotient of the tensor superalgebra $$\bC\big[\tl h_{i,r},\tl h_{j,2r+1}|i\in \I^0_{\ne}, j\in \I^0_{=},r\in\bN\big]\otimes \wtl{\mathbf Y}_\imath^\bs,$$ where $\wtl{\mathbf Y}_\imath^\bs$ is the superalgebra generated by $\tl b_{i,r}$ of parity $|\alpha_i|$ for $i\in \I^0$, $r\in\bN$ subject to the relations \eqref{tobea}. For $\alpha\in\mc R^+$ and $r\in\bN$, define $\tl b_{\alpha,r}$ in the same way as in \eqref{rv} with $b_{k,s}$ replaced by $\tl b_{k,s}$. It follows from \cite[\S2.6--\S2.7]{tsymbaliuk2020shuffle} that $\wtl{\mathbf Y}_\imath^\bs$ is spanned by the ordered super monomials in $\tl b_{\alpha,r}$ for $\alpha\in\mc R^+$ and $r\in\bN$, completing the proof.
\end{proof}

We will see that these ordered monomials indeed form a basis of $\Yi$; see Corollary \ref{thm:pbw}.

\subsection{On Serre relations}\label{app}
In this subsection, we establish that some complicated Serre relations can be deduced effectively from other relations and the finite type Serre relations. Similar reduction steps are commonly used in the study of Yangians and quantum loop algebras. For instance, this is a key part of the minimalistic presentation of Yangians \cite{Levendorski93generators,GNW18}, see also \cite[Proposition 2.7]{GTL16}.

\begin{prop}\label{prop-serre}
Suppose the relations \eqref{qsconj0}--\eqref{qsconj3} hold. Assume further that $\tau i=i\pm 1$, $|\alpha_i|=\bar 0$, and
\beq\label{ppf0}
\big[b_{i,0},[b_{i,0},b_{\tau i,0}]\big]=4b_{i,0},
\eeq
then the relation \eqref{qsconj10} holds as well.
\end{prop}
\begin{proof}
In this case both $b_{i,r}$ and $b_{\tau i,s}$ are even, and hence the proof is exactly the same as that of \cite[Prop. 3.12]{lu2024drinfeld}.
\end{proof}

\begin{prop}\label{propnewserre}
Suppose the relations \eqref{qsconj0}--\eqref{qsconj3} hold. Assume further that $|\alpha_i|=\bar 1$, $|\alpha_{i-1}|=|\alpha_{i+1}|=\bar 0$, and
\beq\label{ppf1}
\big[[b_{i-1,0},b_{i,0}],[b_{i,0},b_{i+1,0}]\big]=0,
\eeq
then the relation \eqref{qsconjnew} holds as well.
\end{prop}
\begin{proof}
Let $N=2\ell$ or $N=2\ell+1$. If $N=2\ell$, then $|\alpha_\ell|=\bar 0$. If $N=2\ell+1$, then $|\alpha_{\ell}|=|\alpha_{\ell+1}|$. Thus our assumptions happen only if $i+1\lle \lfloor \frac{N}2\rfloor$ or $i-1\gge \lfloor \frac{N+1}2\rfloor$. Set $\xi_{j,1}:=h_{j,1}-\frac{1}{2}h_{j,0}^2$, then
\[
[\xi_{j,1},b_{k,r}]=(c_{jk}+c_{\tau j,k})b_{k,r+1}.
\]
Then the relation \eqref{qsconjnew} easily follows by applying $[\xi_{j,1},\cdot\,]$ to \eqref{ppf1} for $j\in \I^0$ such that $|j-i|\lle 1$ and using induction  on $r+s$, cf. \cite{Levendorski93generators}.
\end{proof}

Now let us consider the relation \eqref{qsconj8}. The argument in the rest of this subsection was obtained during the course of the joint work \cite{lu2025affine} with Wang and Zhang, but neither the statement nor the proof of the result below appears in \cite{lu2025affine}. We include the proof here, both for completeness and for its use in the present work.

We fix $i,j\in\I^0$ such that $i=\tau i$, $i\ne j$, and $c_{ij}\ne 0$. Then $b_{i,r}$ are even elements and $c_{ij}=-\ka_i$. 
We assume the relations \eqref{qsconj0}--\eqref{qsconj3} and the finite type Serre relations 
\beq\label{serre0}
\big[b_{i,0},[b_{i,0},b_{j,0}]\big]=-b_{j,0}.
\eeq
Then we shall prove that the general Serre relation \eqref{qsconj8} holds. 

Set
\beq\label{biju}
b_{ij}(u)=[b_{i,0},b_j(u)].
\eeq
Note that \eqref{qsconj3} for the choice of $i,j$ is equivalent to
\[
\begin{split}
(u-v-\tfrac{c_{ij}}{2}) b_i(u) b_j(v)=(u-v+\tfrac{c_{ij}}{2}) b_j(v) b_i(u)+ \big([ b_{i,0}, b_j(v)]-[ b_i(u), b_{j,0}]\big).
\end{split} 
\]
It follows by setting $v=u\pm \tfrac{c_{ij}}{2}$ that  
\beq\label{bi0bj}
[ b_{i}(u), b_{j,0}]= b_{ij}(u-\tfrac{c_{ij}}{2})+c_{ij} b_j(u-\tfrac{c_{ij}}{2}) b_i(u)= b_{ij}(u+\tfrac{c_{ij}}{2})+c_{ij} b_i(u) b_j(u+\tfrac{c_{ij}}{2}).
\eeq
Note that by \eqref{qsconj3}, one also has
\beq\label{bibi0}
[b_i(u),b_{i,0}]=\ka_i\big(b_i(u)^2- h_i(u)+1\big).
\eeq
\begin{lem}
We have
\beq\label{bijbj}
[b_{i,0},b_{ij}(u)]=-b_j(u).
\eeq
\end{lem}
\begin{proof}
Using a similar argument of \cite[\S4.3]{Lu2023drinfeld} by induction and recursively applying $h_{i,1}$ and $\mathpzc{h}_{j,1}:=h_{j,1}-\frac12h_{j,0}^2$, one shows that $\big[b_{i,0},[b_{i,0},b_{j,r}]\big]=-b_{j,r}$. The lemma follows from \eqref{biju}.
\end{proof}

\begin{lem}
We have
\begin{align}
&[ h_i(u), b_{j,0}]= c_{ij} h_i(u) b_j(u+\tfrac{c_{ij}}{2})+c_{ij} b_j(-u-\tfrac{c_{ij}}{2}) h_i(u),\label{hb1}\\
&\frac{u}{u^2-\frac{1}4}\big([ h_i(u), b_{j,1}]+c_{ij} \{ h_i(u), b_{j,0}\}\big)= c_{ij}b_j(u-\tfrac{c_{ij}}{2}) h_i(u)-c_{ij}b_j(-u-\tfrac{c_{ij}}{2}) h_i(u).\label{hb2}
\end{align}
\end{lem}
\begin{proof}
The relation \eqref{qsconj2} for the choice of $i,j$ can be equivalently written as
\begin{align*}
(u^2-v^2)[ h_i(u),b_j(v)]&=c_{ij}v\{ h_i(u), b_j(v)\}+\frac{1}{4}c_{ij}^2[h_i(u),b_j(v)]\\&\quad -[ h_i(u), b_{j,1}]-c_{ij}\{ h_i(u), b_{j,0}\}-v[ h_i(u), b_{j,0}].
\end{align*}
Hence
\begin{align*}
-[ h_i(u), b_{j,1}]&-c_{ij} \{ h_i(u), b_{j,0}\}\\&=(u^2-(v+\tfrac{c_{ij}}{2})^2) h_i(u) b_j(v)-(u^2-(v-\tfrac{c_{ij}}{2})^2) b_j(v) h_i(u)	+v[ h_i(u), b_{j,0}].
\end{align*}
Note that the LHS is independent of $v$.

Setting $v=u+\tfrac{c_{ij}}{2}$, we have
\[
-[ h_i(u), b_{j,1}]-c_{ij}\{ h_i(u), b_{j,0}\}=-(2c_{ij}u+c_{ij}^2) h_i(u) b_j(u+\tfrac{c_{ij}}{2})+(u+\tfrac{c_{ij}}{2})[ h_i(u), b_{j,0}].
\]	
Setting $v=-u-\tfrac{c_{ij}}{2}$, we obtain
\[
-[ h_i(u), b_{j,1}]-c_{ij} \{ h_i(u), b_{j,0}\}=(2c_{ij}u+c_{ij}^2) b_j(-u-\tfrac{c_{ij}}{2}) h_i(u)-(u+\tfrac{c_{ij}}{2})[ h_i(u), b_{j,0}].
\]
Then the first equality follows from these two equations above.

Setting $v=u-\tfrac{c_{ij}}{2}$, we have
\[
-[ h_i(u), b_{j,1}]-c_{ij} \{ h_i(u), b_{j,0}\}=-(2c_{ij}u-c_{ij}^2) b_j(u-\tfrac{c_{ij}}{2}) h_i(u)+(u-\tfrac{c_{ij}}{2})[ h_i(u), b_{j,0}].
\]	
Then the second equality follows from the two equations above.
\end{proof}

\begin{lem}\label{serre0u0}
We have 
\begin{equation}\label{serre0u}
\big[ b_i(u),[ b_{i,0}, b_{j,0}]\big]+\big[ b_{i,0},[ b_{i}(u), b_{j,0}]\big]= b_j(-u-\tfrac{c_{ij}}{2}) h_i(u)-  b_j(u-\tfrac{c_{ij}}{2}) h_i(u).
\end{equation}
\end{lem}
\begin{proof} 
By Jacobi identity and $c_{ij}=-\ka_i$, we get that
\begin{align*}
&\big[ b_{i}(u),[ b_{i,0}, b_{j,0}]\big]+\big[ b_{i,0},[ b_{i}(u), b_{j,0}]\big]\\
=\hskip0.19cm &2\big[[ b_{j,0}, b_{i}(u)], b_{i,0}\big]+\big[ b_{j,0},[ b_{i,0}, b_i(u)]\big]\\
\overset{\eqref{bi0bj}}{\underset{\eqref{bibi0}}{=}} &-\big[ b_{ij}(u-\tfrac{c_{ij}}{2})-\ka_ib_j(u-\tfrac{c_{ij}}{2}) b_i(u), b_{i,0}\big]\\&-\big[ b_{ij}(u+\tfrac{c_{ij}}{2})-\ka_i b_i(u) b_j(u+\tfrac{c_{ij}}{2}), b_{i,0}\big]+\ka_i[ b_{j,0}, h_i(u)- b_i(u)^2]\\
\stackrel{(*)}{=}\hskip0.15cm &- b_j(u-\tfrac{c_{ij}}{2})- \ka_ib_{ij}(u-\tfrac{c_{ij}}{2}) b_i(u)+ b_j(u-\tfrac{c_{ij}}{2})( b_i(u)^2- h_i(u)+1)\\& - b_j(u+\tfrac{c_{ij}}{2})- \ka_ib_i(u) b_{ij}(u+\tfrac{c_{ij}}{2})+ ( b_i(u)^2- h_i(u)+1) b_j(u+\tfrac{c_{ij}}{2})+\ka_i[ b_{j,0}, h_i(u)]\\&+\ka_i( b_{ij}(u-\tfrac{c_{ij}}{2})-\ka_i b_j(u-\tfrac{c_{ij}}{2}) b_i(u)) b_i(u)+ \ka_ib_i(u)( b_{ij}(u+\tfrac{c_{ij}}{2})-\ka_i b_i(u) b_j(u+\tfrac{c_{ij}}{2}))\\
=\hskip0.19cm &- b_j(u-\tfrac{c_{ij}}{2}) h_i(u)-  h_i(u) b_j(u+\tfrac{c_{ij}}{2})+\ka_i[ b_{j,0}, h_i(u)]\\
\stackrel{\eqref{hb1}}{=}&~ b_j(-u-\tfrac{c_{ij}}{2}) h_i(u)-  b_j(u-\tfrac{c_{ij}}{2}) h_i(u),
\end{align*}
where in $(*)$ we applied \eqref{biju}--\eqref{bijbj}.
\end{proof}

\begin{rem}
Since $ h_i(u)$ is even, it follows from \eqref{hb1} that
\[
[ h_i(u), b_{j,0}]=- \ka_ih_i(u) b_j(-u+\tfrac{c_{ij}}2)- \ka_ib_j(u-\tfrac{c_{ij}}2) h_i(u).
\]
Using the above equation instead of \eqref{hb1} in the proof of Lemma \ref{serre0u0}, we find that
\beq\label{serre0u-new}
\big[ b_i(u),[ b_{i,0}, b_{j,0}]\big]+\big[ b_{i,0},[ b_{i}(u), b_{j,0}]\big]= h_i(u) b_j(-u+\tfrac{c_{ij}}2)-h_i(u) b_j(u+\tfrac{c_{ij}}2).
\eeq
\end{rem}

\begin{prop}\label{prop:Serre}
Suppose the relations \eqref{qsconj0}--\eqref{qsconj3} and \eqref{serre0} hold, then the Serre relation \eqref{qsconj8} also holds.
\end{prop}
\begin{proof}
By \eqref{hb2} and \eqref{serre0u}, we find that
\[
\big[ b_i(u),[ b_{i,0}, b_{j,0}]\big]+\big[ b_{i,0},[ b_{i}(u), b_{j,0}]\big]=\frac{u}{u^2-\frac14} \big(\ka_i[ h_i(u), b_{j,1}]- \{ h_i(u), b_{j,0}\}\big).
\]
Expanding $(4u^2-1)^{-1}$ as a power series in $u^{-1}$ and comparing coefficients, one shows the Serre relations \eqref{qsconj8} for the case $k_1=r=0$ and $k_2\in\mathbb N$, which corresponds to \cite[Claim 1 in \S4.3]{Lu2023drinfeld}. Then as argued in \cite[Claims 2-3 in \S4.3]{Lu2023drinfeld}, one proves that
\[
\mathbb S_{ij}(k_1+1,k_2;r)=-\mathbb S_{ij}(k_1,k_2+1;r).
\]
Hence the Serre relation \eqref{qsconj8} for general case is reduced to the special case $k_1=r=0$ and $k_2\in\mathbb N$ proved above.
\end{proof}

From the proof of Proposition \ref{prop:Serre}, we also have the following.
\begin{cor}\label{implycor}
The relations \eqref{qsconj2} and \eqref{serre0u} imply the Serre relations \eqref{qsconj8}. 
\end{cor}

\section{Twisted super Yangians in R-matrix presentation}\label{sec:R}
In this section, we recall the basics for the twisted super Yangians $\Y^\bs$ of quasi-split type A (under the name \textit{reflection superalgebra})  defined in the R-matrix presentation from \cite{Molev2002reflection,Ragoucy2007analytical}, cf. also \cite{lu2023twisted}.

\subsection{Yangians}
\label{subsec:Y}
We start with recalling the basic theory of super Yangian $\rY(\gl^\bs)$ from \cite{Nazarov1991berezinian}. 

In this subsection, we do not require $\bs$ to be symmetric. Let $\C^\bs$ be the vector superspace with a basis $\mathbf v_i$ for $i\in \I$ such that $|\mathbf v_i|=|i|$. Let $E_{ij}\in \End(\C^\bs)$ be the linear operators such that $E_{ij}\mathbf v_k=\delta_{jk}\mathbf v_i$ for $i,j,k\in\I$.
\begin{dfn}
The \textit{super Yangian} $\rY(\gl^{\bs})$ corresponding to the Lie superalgebra $\gl^\bs$ is a unital associative superalgebra with generators $t_{ij}^{(r)}$ of parity $|i|+|j|$, where $i,j\in\I$ and $r\in\bZ_{>0}$, and the defining relations written in terms of the generating series
\[
t_{ij}(u)=  \delta_{ij}+t_{ij}^{(1)}u^{-1}+t_{ij}^{(2)}u^{-2}+\cdots
\]
are given by the relations 
\beq\label{Trel}
(u-v)[t_{ij}(u),t_{kl}(v)]=(-1)^{|i||j|+|i||k|+|j||k|}\big(t_{kj}(u)t_{il}(v)-t_{kj}(v)t_{il}(u)\big).
\eeq
\end{dfn}

The super Yangian $\rY(\gl^\bs)$ has the following R-matrix presentation. Let $R(u)$ be the Yang R-matrix
\begin{align} \label{Ru}
R(u)=1-\frac{P}{u}\in \End(\bC^\bs\otimes \bC^\bs)[u^{-1}],\quad \text{where} \quad P=\sum_{i,j=1}^N \ka_jE_{ij}\otimes E_{ji},
\end{align}
and
\[
T(u)=\sum_{i,j=1}^N t_{ij}(u)\otimes E_{ij}(-1)^{|i||j|+|j|}\in \rY(\gl^\bs)[[u^{-1}]]\otimes\End(\bC^\bs). 
\]
Then the defining relations of $\rY(\gl^\bs)$ can be written as
\be
R(u-v)T_1(u)T_2(v)=T_2(v)T_1(u)R(u-v).
\ee

Note that the Yang R-matrix satisfies the Yang-Baxter equation
\beq\label{ybeq}
R_{12}(u-v)R_{13}(u)R_{23}(v)=R_{23}(v)R_{13}(u)R_{12}(u-v).
\eeq
The super flip operator $P$ has the property $P(\mathbf v_i\otimes \mathbf v_j)=(-1)^{|i||j|}\mathbf v_j\otimes \mathbf v_i$.

Let $g(u)$ be any formal power series in $u^{-1}$ with leading term $1$,
\[
g(u)=1+g_1u^{-1}+g_2u^{-2}+\cdots\in \bC[[u^{-1}]].
\]
There is an automorphism of $\rY(\gl^\bs)$ defined by
\beq\label{eq:mu_f-A}
\mathcal M_{g(u)}^\bs:T(u)\to g(u)T(u).
\eeq

The super Yangian for $\mathfrak{sl}^\bs$ is the subalgebra $\rY(\mathfrak{sl}^\bs)$ of $\rY(\gl^\bs)$ which consists of all elements stable under all the automorphisms of the form \eqref{eq:mu_f-A}.

Consider the filtration on $\rY(\gl^\bs)$ obtained by setting
\begin{align}  \label{filter:Y}
\deg t_{ij}^{(r)}=r-1
\end{align}
for every $r\gge 1$. Denote by $\mathrm{gr}\rY(\gl^\bs)$ the associated graded superalgebra. We write $\bar t_{ij}^{(r)}$ the image of $t_{ij}^{(r)}$ in $\mathrm{gr}\rY(\gl^\bs)$. Let $\gl^\bs[z]$ be the polynomial current superalgebra of $\gl^\bs$ in the indeterminate $z$. Then the map
\beq\label{eq:cl-limitA}
\mathrm{U}(\gl^\bs[z])\to \mathrm{gr}\rY(\gl^\bs), \qquad \ka_ie_{ij} z^{r}\mapsto \bar t_{ij}^{(r+1)},
\eeq
induces a Hopf superalgebra isomorphism.

We collect a few facts about the inverse of $T(u)$ of $\rY(\gl^\bs)$. Define the series $\tl t_{ij}(u)$, whose coefficients $\tl t_{ij}^{(r)}$ are in $\rY(\gl^\bs)$,
$$
\tl t_{ij}(u):=\delta_{ij}+\sum_{r>0}\tl t_{ij}^{(r)}u^{-r}
$$
by
$$
\wtl T(u):=\big(T(u)\big)^{-1}=\sum_{i,j=1}^N \tl t_{ij}(u)\otimes E_{ij}(-1)^{|i||j|+|j|}.
$$
Then
\beq\label{T'ij}
\tl t_{ij}(u)=\delta_{ij}+\sum_{k>0} (-1)^k\sum_{a_1,\cdots,a_{k-1}=1}^N t_{ia_1}^\circ(u)t_{a_1a_2}^\circ(u)\cdots t_{a_{k-1}j}^\circ(u),
\eeq
where $t_{ij}^\circ(u)=t_{ij}(u)-\delta_{ij}$. In particular, by taking the coefficient of $u^{-r}$, for $r\gge 1$, one obtains 
\beq\label{t'ijr}
\tl t_{ij}^{(r)}=\sum_{k=1}^r (-1)^k\sum_{a_1,\cdots,a_{k-1}=1}^N\sum_{r_1+\cdots+r_k=r}t_{ia_1}^{(r_1)}t_{a_1a_2}^{(r_2)}\cdots t_{a_{k-1}j}^{(r_k)},
\eeq
where $r_i$ for $1\lle i\lle k$ are positive integers.

\subsection{Twisted super Yangians}\label{subsec:tY}
Recall that $i'=N+1-i$ for $1\lle i\lle N$ and $\mathfrak s$ is symmetric. Let $G=(g_{ij})$ be the $N\times N$ even matrix defined by $g_{ij}=\delta_{ij'}$. For any $N\times N$ super matrix $M=(m_{ij})$, define
\[
M'=G M G^{-1}=(m_{i'j'}).
\]
In particular, we have the modified R-matrix,
\beq\label{twistedR}
R'(u)=G_1R(u)G_1=G_2R(u)G_2.
\eeq

The following twisted super Yangians were specific reflection (super)algebras \cite{Sklyanin1988Boundary} introduced in \cite{Molev2002reflection,Ragoucy2007analytical}; see also \cite{Belliard2009nested,lu2023twisted}.

\begin{dfn}\label{def:R-Y}
The  \textit{twisted super Yangian $\Y^\bs$ of quasi-split type A} is a unital associative superalgebra with generators $x_{ij}^{(r)}$ of parity $|i|+|j|$, where $i,j\in\I$ and $r\in\bZ_{>0}$, and the defining relations written in terms of the generating series
\beq\label{siju}
x_{ij}(u)=  \delta_{ij}+x_{ij}^{(1)}u^{-1}+x_{ij}^{(2)}u^{-2}+\cdots
\eeq
are given by the \textit{quaternary} relations 
\beq\label{bcom}
\begin{split}
[x_{ij}(u),x_{kl}(v)]&= \frac{(-1)^{|i||j|+|i||k|+|j||k|}}{u-v}(x_{kj}(u)x_{il}(v)-x_{kj}(v)x_{il}(u))\\
&  + \frac{(-1)^{|i||j|+|i||k|+|j||k|}}{u+v}\Big(\delta_{kj'}\sum_{a=1}^N x_{ia'}(u)x_{al}(v)-\delta_{il'}\sum_{a=1}^N x_{ka'}(v)x_{aj}(u)\Big)\\
& - \frac{\delta_{ij'}}{u^2-v^2}\Big(\sum_{a=1}^N x_{ka'}(u)x_{al}(v)-\sum_{a=1}^N x_{ka'}(v)x_{al}(u)\Big)
\end{split}
\eeq
and the \textit{unitary} condition
\beq\label{bunit}
\sum_{a=1}^N x_{ia'}(u)x_{aj}(-u)=\delta_{ij'}.
\eeq
\end{dfn}
Define the operator $X(u)\in  \Y^\bs[[u^{-1}]]\otimes\End(\bC^\bs)$,
\[
X(u)=\sum_{i,j=1}^N x_{ij}(u)\otimes E_{ij}(-1)^{|i||j|+|j|}.
\]
Then the defining relations of $\Y^\bs$ are given by
\beq\label{quamat}
R(u-v)X_1(u)R'(u+v)X_2(v)=X_2(v)R'(u+v)X_1(u)R(u-v),
\eeq
\beq\label{unimat}
X(u)X'(-u)=\mathbf 1^\bs,
\eeq
where $\mathbf 1^\bs$ is the identity matrix in $\End(\C^\bs)$.

It is convenient to work with the extended twisted super Yangians defined below instead of twisted super Yangians. By abuse of notations, we shall keep using the same notations for the twisted super Yangians and extended twisted super  Yangians 
of various elements such as $x_{ij}(u)$ and $X(u)$, etc.
\begin{dfn}\label{eradef}
The \textit{extended twisted super Yangian $\scrX^\bs$ of quasi-split type A} is the unital associative superalgebra  with generators $x_{ij}^{(r)}$ of parity $|i|+|j|$, where $i,j\in\I$ and $r\in\bZ_{>0}$ satisfying the quaternary relations \eqref{quamat}, where $x_{ij}(u)$ is again given by \eqref{siju}.
\end{dfn}

Sometimes, we shall use super Yangians and twisted super Yangians whose ranks are smaller than $\rY(\gl^\bs)$ and $\Y^\bs$. In our situation, they are associated to subsequences of the parity sequence $\bm\ka$. 

For $1\lle i\lle j\lle N$ and a parity sequence $\bs=(\ka_1,\ka_2,\cdots,\ka_N)$, we use the notation
\[
\bs_{[i,j]}:=(\ka_i,\ka_{i+1},\cdots,\ka_{j})
\]
and similar notations in terms of open intervals. 

For a fixed symmetric parity sequence $\bs$ and $1\lle m \lle \lfloor \frac{N-1}{2}\rfloor$, the parity sequences $\bs_{(m,m')}$ and $\bs_{[m,m']}$ are symmetric as well. We further use the following
\beq\label{ysbs-sub}
\rY(\gl^\bs_{[m]}):=\rY(\gl^{\bs_{[1,m]}}),\quad \scrX^\bs_{(m,m')}:=\scrX^{\bs_{(m,m')}},\quad \scrX^\bs_{[m,m']}:=\scrX^{\bs_{[m,m']}}.
\eeq

\subsection{Basic properties}\label{sec:basics}
Let $\wtl X(u)=X(u)^{-1}=(\tl x_{ij}(u))$, i.e.
\[
\wtl X(u)=\sum_{i,j=1}^N \tl x_{ij}(u)\otimes E_{ij}(-1)^{|i||j|+|j|}.
\]
\begin{prop}[{\cite[Prop.~3.2]{lu2023twisted}}]\label{Binv}
In the extended twisted super Yangian $\scrX^\bs$, the product $X(u)X'(-u)$ is a scalar matrix,
\beq\label{cudef}
X(u)X'(-u)=X'(-u)X(u)=c(u)\mathbf 1^\bs,
\eeq
where $c(u)$ is an even series in $u^{-1}$ whose coefficients are central in $\scrX^\bs$. In particular, we have $x_{i'j'}(-u)=c(u)\tl x_{ij}(u)$.
\end{prop}

It is known that $\Y^\bs$ can be identified as a subalgebra of $\rY(\gl^\bs)$, see e.g. \cite{kettle2023orthosymplectic,lu2023twisted,bagnoli2023double}. Specifically, the map
\beq\label{inc}
X(u)\mapsto T(u)\wtl T'(-u)
\eeq
defines a superalgebra embedding $\Y^\bs\hookrightarrow \rY(\gl^\bs)$.
Moreover, there is a filtration on $\Y^\bs$ inherited from the one \eqref{filter:Y} on $\rY(\gl^\bs)$ such that $\deg x_{ij}^{(r)}=r-1$. Let $\mathcal F_s(\Y^\bs)$ be the subspace of $\Y^\bs$ spanned by elements of degree $\lle s$. Then
\begin{align}
\label{filter:B}
    \mathcal F_0(\Y^\bs) \subset \mathcal F_1(\Y^\bs) \subset \mathcal F_2(\Y^\bs) \subset \ldots, 
    \qquad\qquad \Y^\bs =\bigcup_{s\gge 0}  \mathcal F_s(\Y^\bs).
\end{align}
Denote by $\gr\, \Y^\bs$ the associated graded superalgebra. Let $\bar x_{ij}^{(r)}$ be the image of $x_{ij}^{(r)}$ in the $(r-1)$-st component of $\gr\,\Y^\bs$. Then by \eqref{t'ijr}
\beq\label{image-quo}
\bar x_{ij}^{(r)}=\bar t_{ij}^{(r)}-(-1)^r\bar t_{i'j'}^{(r)}.
\eeq

Let $\vartheta$ be the involution of $\gl^\bs$ defined by
\[
\vartheta: \gl^\bs\longrightarrow \gl^\bs,\quad e_{ij}\mapsto e_{i'j'}.
\]
Extend this involution to $\gl^\bs[z]$ by sending $g  z^r$ to $\vartheta(g) (-z)^r$ for $g\in \gl^\bs$ and $r\in\bN$. Let $\gl^\bs[z]^\vartheta$ be the fixed point subalgebra of $\gl^\bs[z]$ under the involution $\vartheta$. Then it is  known that the map
\beq\label{isogr}
\mathrm{U}(\gl^\bs[z]^\vartheta)\longrightarrow \gr\,\Y^\bs,\qquad \ka_i\big(e_{ij}+(-1)^re_{i'j'}\big) z^{r} \mapsto \bar x_{ij}^{(r+1)}
\eeq
induces a superalgebra isomorphism, see \cite[Prop.~3.6]{lu2023twisted} and cf. \eqref{eq:cl-limitA} and \eqref{image-quo}. 

By restriction, we can also define $\mathfrak{sl}^\bs[z]^\vartheta$ and $\mathrm U(\mathfrak{sl}^\bs[z]^\vartheta)$.

\section{Gauss decomposition}\label{sec:GDmain}

In this section, we formulate and study the Gauss decomposition for twisted super Yangians. Using the Gaussian generators, we establish in Theorem~\ref{main2} an isomorphism between $\Yi$ introduced in Definition \ref{deftY} and the special twisted super Yangian $\SY^\bs$.

\subsection{Quasi-determinants and Gauss decomposition}\label{sec:GD}
We shall also need the quasi-determinant presentation, see \cite{Gelfand2005quasi}, of Drinfeld current generating series in terms of the R-matrix generating series. 

Let $\mathsf X$ be a square (super)matrix over a ring with identity. Let $\mathsf X^{ij}$ be the submatrix of $\mathsf X$ obtained by deleting the $i$-th row and $j$-th column. Let $\mathsf R_{i}^j$ be the row matrix obtained from the $i$-th row of $\mathsf X$ by removing $\mathsf x_{ij}$ and $\mathsf C_{j}^i$ be the column matrix obtained from the $j$-th column of $\mathsf X$ by deleting $\mathsf x_{ij}$. Suppose $\mathsf X^{ij}$ is invertible.  Then the $(i,j)$-th
\emph{quasi-determinant} of $\mathsf X$ is defined by the first formula below and denoted graphically by the boxed notation (cf. \cite[\S1.10]{Molev2007book}):
\begin{equation*}
\vert \mathsf X\vert _{ij} \stackrel{\text{def}}{=} \mathsf x_{ij}-\mathsf R_i^j\big(\mathsf X^{ij}\big)^{-1}\mathsf C_j^i = \left\vert  \begin{array}{ccccc} \mathsf x_{11} & \cdots & \mathsf x_{1j} & \cdots & \mathsf x_{1n}\\
&\cdots & & \cdots&\\
\mathsf x_{i1} &\cdots &\mybox{$\mathsf x_{ij}$} & \cdots & \mathsf x_{in}\\
& \cdots& &\cdots & \\
\mathsf x_{n1} & \cdots & \mathsf x_{nj}& \cdots & \mathsf x_{nn}
\end{array} \right\vert .
\end{equation*}

By \cite[Theorem 4.96]{Gelfand2005quasi}, the matrix $X(u)$, for both $\scrX^\bs$ and $\Y^\bs$, has the following Gauss decomposition:
$$
X(u) = F(u) D(u) E(u)
$$
for unique matrices of the form
\begin{equation*}
D(u) = \left[ \begin{array}{cccc} d_1 (u) & &\cdots & 0\\
& d_2 (u) &  &\vdots\ \\
\vdots & &\ddots &\\
0 &\cdots &  &d_{N} (u)
\end{array} \right],
\end{equation*}
\begin{equation*}
E(u)=\left[ \begin{array}{cccc} \!\!1 &e_{12}(u) &\cdots & e_{1N}(u)\!\!  \\
&\ddots & &e_{2N}(u) \!\! \\
& &\ddots & \vdots\\
0 & & &1 
\end{array} \right],\qquad 
F(u) = \left[ \begin{array}{cccc} \!\!1 & &\cdots &0\!\!\\
f_{21}(u) &\ddots & &\vdots\\
\vdots & & \ddots& \\
\!\!f_{N1}(u) & f_{N2}(u) &\cdots &1\!
\end{array} \right],
\end{equation*}
where the matrix entries are defined in terms of quasi-determinants:
\begin{eqnarray}
d_i (u) &=& \left\vert  \begin{array}{cccc} x_{11}(u) &\cdots &x_{1,i-1}(u) &x_{1i}(u) \\
\vdots &\ddots & &\vdots \\
x_{i1}(u) &\cdots &x_{i,i-1}(u) &\mybox{$x_{ii}(u)$}
\end{array} \right\vert, 
\qquad \tl d_i(u)=d_i(u)^{-1}, \label{gd1}
\\
e_{ij}(u) &=&\tl d_i (u) \left\vert  \begin{array}{cccc} x_{11}(u) &\cdots & x_{1,i-1}(u) & x_{1j}(u) \\
\vdots &\ddots &\vdots & \vdots \\
x_{i-1,1}(u) &\cdots &x_{i-1,i-1}(u) & x_{i-1,j}(u)\\
x_{i1}(u) &\cdots &x_{i,i-1}(u) &\mybox{$x_{ij}(u)$}
\end{array} \right\vert,\label{gd2}
\\
f_{ji}(u) &=& \left\vert  \begin{array}{cccc} x_{11}(u) &\cdots &x_{1, i-1}(u) & x_{1i}(u) \\
\vdots &\ddots &\vdots &\vdots \\
x_{i-1,1}(u) &\cdots &x_{i-1,i-1}(u) &x_{i-1,i}(u)\\
x_{j1}(u) &\cdots &x_{j, i-1}(u) &\mybox{$x_{ji}(u)$} 
\end{array} \right\vert\, 
\tl d_{i}(u).\label{gd3}
\end{eqnarray}
The Gauss decomposition can also be written component-wise as, for $i<j$, 
\begin{align}
x_{ii}(u)&=d_i(u)+\sum_{k<i}f_{ik}(u)d_k(u)e_{ki}(u),\nonumber\\
x_{ij}(u)&=d_i(u)e_{ij}(u)+\sum_{k<i}f_{ik}(u)d_k(u)e_{kj}(u),\label{eq:sij-Gauss}\\
x_{ji}(u)&=f_{ji}(u)d_i(u)+\sum_{k<i}f_{jk}(u)d_k(u)e_{ki}(u).\nonumber
\end{align}

We further denote
\begin{align}
    \label{gauss-gen}
e_{ij}(u) &=\sum_{r\gge 1}e_{ij}^{(r)}u^{-r},\quad f_{ji}(u)=\sum_{r\gge 1}f_{ji}^{(r)}u^{-r},\quad d_k(u)=1+\sum_{r\gge 1}d_{k}^{(r)}u^{-r}.
\\
e_i(u) &=\sum_{r\gge 1}e_{i}^{(r)}u^{-r}=e_{i,i+1}(u),\quad f_i(u)=\sum_{r\gge 1}f_{i}^{(r)}u^{-r}=f_{i+1,i}(u),\quad 1\lle i<N.
\end{align}

Set 
\beq\label{edfinv}
\begin{split}
&\wtl D(u)=D(u)^{-1}=\sum_{1\lle i\lle N} \tl d_{i}(u)\otimes E_{ii},
\\
&\wtl E(u)=E(u)^{-1}=\sum_{1\lle i<j\lle N}\tl e_{ij}(u)\otimes E_{ij}(-1)^{|i||j|+|j|},\\
&\wtl F(u)=F(u)^{-1}=\sum_{1\lle i<j\lle N}\tl f_{ji}(u)\otimes E_{ji}(-1)^{|i||j|+|i|}.
\end{split}
\eeq
Then we have 
\beq\label{eq:def-tilde-e-f}
\begin{split}
&\tl e_{ij}(u)=\sum_{i=i_0<i_1<\cdots<i_s=j}(-1)^s e_{i_0i_1}(u)e_{i_1i_2}(u)\cdots e_{i_{s-1}i_s}(u),\\
&\tl f_{ji}(u)=\sum_{i=i_0<i_1<\cdots<i_s=j}(-1)^s f_{i_{s}i_{s-1}}(u)\cdots f_{i_2i_1}(u) f_{i_1i_0}(u).
\end{split}
\eeq

\subsection{A homomorphism $\scrX^\bs_{(m,m')}\to \scrX^\bs$}

Unlike the case of type AI in \cite{Lu2023drinfeld}, the commutator relations \eqref{bcom} for type AIII involve summation. Consequently, there is no obvious (natural) embedding from $\scrX^{\tl\bs}$ to $\scrX^\bs$ for a symmetric parity subsequence $\tl\bs$ of $\bs$. From the viewpoint of Satake diagrams for quantum symmetric pairs \cite{Letzter1999coideal,Kolb2014quantum} and their super analogues \cite{ShenWang2024quantum}, one still expects a homomorphism from $\scrX^\bs_{(m,m')}\to \scrX^\bs$; recall the definition of $\scrX^\bs_{(m,m')}$ from \eqref{ysbs-sub}. The main goal of this section is to construct such a homomorphism, following similar strategy of \cite[\S3]{Jing18iso}.

For $2\lle  i\lle 2' (=N-1)$, we have 
\beq\label{emrel1}
x_{11}(u+\ka_1)x_{i 1}(u)=x_{i1}(u+\ka_1)x_{11}(u).
\eeq
Therefore
\begin{align*}
\begin{vmatrix}x_{11}(u) & x_{1j}(u)\\
x_{i1}(u) & \mybox{$x_{ij}(u)$}\end{vmatrix}&=x_{ij}(u)-x_{i1}(u)x_{11}(u)^{-1}x_{1j}(u)\\
&=x_{11}(u+\ka_1)^{-1}\big(x_{11}(u+\ka_1)x_{ij}(u)-x_{i1}(u+\ka_1)x_{1j}(u)\big).
\end{align*}
Set 
\beq\label{newt}
\mc T_{ij}(u)=x_{11}(u+\ka_1)x_{ij}(u)-x_{i1}(u+\ka_1)x_{1j}(u)=x_{11}(u+\ka_1)\begin{vmatrix}x_{11}(u) & x_{1j}(u)\\
x_{i1}(u) & \mybox{$x_{ij}(u)$}\end{vmatrix}
\eeq
and introduce
\beq\label{gamma}
\begin{split}
\Gamma(u)&=\sum_{a_i,b_i}E_{a_1b_1}\otimes E_{a_2b_2}\otimes \Gamma_{b_1b_2}^{a_1a_2}(u)\\
&=R_{12}(\ka_1)X_1(u+\ka_1){R}_{12}'(2u+\ka_1)X_2(u)=X_2(u){R}_{12}'(2u+\ka_1)X_1(u+\ka_1)R_{12}(\ka_1),
\end{split}
\eeq
where the last equality follows from \eqref{quamat}. 
\begin{lem}\label{Blem}
We have
\begin{enumerate}
    \item  $[x_{11}(u),\mc T_{ij}(v)]=0$, $2\lle i,j\lle 2'$;
\item $\Gamma_{1j}^{1i}(u)=\mc T_{ij}(u)(-1)^{|i||j|+|j|}$, $2\lle i,j\lle 2'$;
\item $\Gamma_{j_1j_2}^{i_1i_2}(u)=-\ka_1(-1)^{|i_1||i_2|+|i_1||j_1|+|i_2||j_1|}\Gamma_{j_1j_2}^{i_2i_1}(u)=-\ka_1(-1)^{|j_1||j_2|+|j_1||i_2|+|j_2||i_2|}\Gamma_{j_2j_1}^{i_1i_2}(u)$.
\end{enumerate}
\end{lem}
\begin{proof}
(1) Note that by \eqref{bcom} we have $[x_{11}(u),x_{11}(v)]=0$. It follows from \eqref{bcom} that
\begin{align*}
&[x_{11}(u),\mc T_{ij}(v)]=[x_{11}(u),x_{11}(v+\ka_1)x_{ij}(v)-x_{i1}(v+\ka_1)x_{1j}(v)]\\
=&\,\frac{\ka_1}{u-v}x_{11}(v+\ka_1)\big(x_{i1}(u)x_{1j}(v)- x_{i1}(v)x_{1j}(u) \big)-\frac{\ka_1}{u-v-\ka_1}\big(x_{i1}(u)x_{11}(v+\ka_1)\\
&\qquad   -x_{i1}(v+\ka_1)x_{11}(u)\big)x_{1j}(v)-\frac{\ka_1}{u-v}x_{i1}(v+\ka_1)\big(x_{11}(u)x_{1j}(v)- x_{11}(v)x_{1j}(u) \big).
\end{align*}
Due to \eqref{emrel1}, it suffices to show that
\begin{align*}
\frac{1}{u-v}x_{11}(v+\ka_1)x_{i1}(u)&-\frac{1}{u-v-\ka_1}x_{i1}(u)x_{11}(v+\ka_1)\\
&+\frac{\ka_1}{(u-v)(u-v-\ka_1)}x_{i1}(v+\ka_1)x_{11}(u)=0
\end{align*}
which is equivalent to
\[
(v+\ka_1-u)[x_{11}(v+\ka_1),x_{i1}(u)]=\ka_1\big(x_{i1}(v+\ka_1)x_{11}(u)-x_{i1}(u)x_{11}(v+\ka_1)\big).
\]
This follows directly from \eqref{bcom}.

(2) Computing $\Gamma_{1j}^{1i}(u)$, $2\lle i,j\lle 2'$, using the definition \eqref{gamma}, one finds that it is given by 
$$
\big(x_{11}(u+\ka_1)x_{ij}(u)-x_{i1}(u+\ka_1)x_{1j}(u)\big)(-1)^{|i||j|+|j|}
$$ 
which coincides with $\mc T_{ij}(u)(-1)^{|i||j|+|j|}$ in \eqref{newt}.

(3) Note that $(1-\ka_1P_{12}) R_{12}(\ka_1)=2R_{12}(\ka_1)$. Thus $R_{12}(\ka_1)$ remains unchanged when multiplying by $(1-\ka_1P_{12})/2$ from the left. Then applying multiplication by $(1-\ka_1P_{12})/2$ from the left to \eqref{gamma}, one derives $\Gamma_{j_1j_2}^{i_1i_2}(u)=-\ka_1(-1)^{|i_1||i_2|+|i_1||j_1|+|i_2||j_1|}\Gamma_{j_1j_2}^{i_2i_1}(u)$. 

To prove $\Gamma_{j_1j_2}^{i_1i_2}(u)=-\ka_1(-1)^{|j_1||j_2|+|j_1||i_2|+|j_2||i_2|}\Gamma_{j_2j_1}^{i_1i_2}(u)$, using
\[
R_{12}(\ka_1)X_1(u+\ka_1){R}_{12}'(2u+\ka_1)X_2(u)=X_2(u){R}_{12}'(2u+\ka_1)X_1(u+\ka_1)R_{12}(\ka_1)
\]
from \eqref{quamat}, one exploits the same approach with  multiplication by $(1-\ka_1P_{12})/2$ from the right to \eqref{gamma}.
\end{proof}

We will need the following simplified expression of \eqref{ybeq} when $v=u-\ka_1$.
\begin{lem}\label{Rsim}
We have the following relations,
\be
R_{12}(\ka_1)R_{13}(u)R_{23}(u-\ka_1)=R_{12}(\ka_1)\Big(1-\frac{P_{13}+P_{23}}{u-\ka_1}\Big),
\ee
\be
R_{23}(u-\ka_1)R_{13}(u)R_{12}(\ka_1)=\Big(1-\frac{P_{13}+P_{23}}{u-\ka_1}\Big)R_{12}(\ka_1).
\ee
\end{lem}

\begin{prop}[{cf.~\cite[Lem.~3.6]{Jing18iso}}]\label{embedB}
The map $x_{ij}(u)\mapsto \mc T_{ij}(u)$, $2\lle i,j\lle 2'$, defines a homomorphism $\scrX^\bs_{[2,2']}\to \scrX^\bs$.
\end{prop}
\begin{proof}We first introduce some shorthand notations. 
Let $u$ and $v$ be parameters. Set $a=u-v$, $\tl a=u+v$. We have the following equality in the superalgebra $\scrX^\bs\otimes \End(\bC^\bs)^{\otimes 4}$, 
\beq\label{pf1}
\begin{split}
&R_{23}(a-\ka_1)R_{13}(a)R_{24}(a)R_{14}(a+\ka_1)\Gamma_{12}(u)\\
&\qquad \times R'_{14}(\tl a+\ka_1)R'_{24}(\tl a)R'_{13}(\tl a+2\ka_1)R'_{23}(\tl a+\ka_1)\Gamma_{34}(v)\\
&\qquad\qquad=\Gamma_{34}(v)R'_{23}(\tl a+\ka_1)R'_{13}(\tl a+2\ka_1)R'_{24}(\tl a)R'_{14}(\tl a+\ka_1)\\
&\qquad\qquad\qquad\qquad \times\Gamma_{12}(u) R_{14}(a+\ka_1)R_{24}(a)R_{13}(a)R_{23}(a-\ka_1).
\end{split}
\eeq
This follows from the Yang-Baxter equation \eqref{ybeq}, where we also used \eqref{twistedR}, and the relations \eqref{quamat}. We shall rewrite both sides of \eqref{pf1} by Lemma \ref{Rsim} and then equate certain matrix elements.

Consider the right hand side of \eqref{pf1}. Applying \eqref{gamma}, \eqref{quamat}, and Lemma \ref{Rsim}, we have
\begin{align*}
&\Gamma_{34}(v)R'_{23}(\tl a+\ka_1)R'_{13}(\tl a+2\ka_1)R'_{24}(\tl a)R'_{14}(\tl a+\ka_1)\\
&\quad\times\Gamma_{12}(u) R_{14}(a+\ka_1)R_{24}(a)R_{13}(a)R_{23}(a-\ka_1)\\
&\quad =\Gamma_{34}(v)\Big(1-\frac{P_{13}'+P_{23}'}{\tl a+\ka_1}\Big)\Big(1-\frac{P_{14}'+P_{24}'}{\tl a}\Big)\Gamma_{12}(u)\Big(1-\frac{P_{14}+P_{24}}{a}\Big)\Big(1-\frac{P_{13}+P_{23}}{a-\ka_1}\Big).
\end{align*}
Then we apply the operator above to a basis vector of the form $\mathbf v_1\otimes \mathbf v_j\otimes \mathbf v_1\otimes \mathbf v_l$ for certain $j,l\in\{2,\cdots,2'\}$. The application of the factor
\[
\Big(1-\frac{P_{14}+P_{24}}{a}\Big)\Big(1-\frac{P_{13}+P_{23}}{a-\ka_1}\Big)
\]
gives
\beq\label{pf2}
\begin{split}
&\frac{a-2\ka_1}{a-\ka_1}\Big(\mathbf v_1\otimes \mathbf v_j\otimes \mathbf v_1\otimes \mathbf v_l  -\frac{1}{a}\mathbf v_1\otimes \mathbf v_l\otimes \mathbf v_1\otimes \mathbf v_j(-1)^{|1||j|+|1||l|+|j||l|}\Big)\\
&- \frac{1}{a-\ka_1}\mathbf v_{1}\otimes \mathbf v_1\otimes \mathbf v_j\otimes \mathbf v_l(-1)^{|1||l|}-\frac{\ka_1(a-2\ka_1)}{a(a-\ka_1)}\mathbf v_l\otimes \mathbf v_j\otimes \mathbf v_1\otimes \mathbf v_1(-1)^{|1||j|+|j||l|} \\
&+ \frac{1}{a(a-\ka_1)}\Big(\mathbf v_l\otimes \mathbf v_1\otimes \mathbf v_j\otimes \mathbf v_1(-1)^{|1|+|j||l|}+\mathbf v_1\otimes \mathbf v_l\otimes \mathbf v_j\otimes \mathbf v_1(-1)^{|1||l|+|j||l|})\Big).
\end{split}
\eeq
It follows from Lemma \ref{Blem} (3) that $\Gamma_{11}^{a_1a_2}(u)=0$ and hence a further application of $\Gamma_{12}(u)$ annihilates the first term in the second line of \eqref{pf2}. Similarly, consider the action of $\Gamma_{12}(u)$ on the last line of \eqref{pf2}, we obtain
\begin{align*}
\Gamma_{l1}^{a_1a_2}(u)\otimes \mathbf v_{a_1}\otimes \mathbf v_{a_2}\otimes \mathbf v_j\otimes \mathbf v_1(-1)^{|a_2||l|+|1||l|+|1|+|j||l|}&\\
+\,\Gamma_{1l}^{a_1a_2}(u)\otimes \mathbf v_{a_1}\otimes \mathbf v_{a_2}\otimes \mathbf v_j\otimes \mathbf v_1(-1)^{|a_2||1|+|j||l|}&.
\end{align*}
It follows from Lemma \ref{Blem}\,(3) again that the above sum vanishes. We consider the further application of the remaining factors acting on the second term in the second line of \eqref{pf2}. The application of $\Gamma_{12}(u)$ on $\mathbf v_l\otimes \mathbf v_j\otimes \mathbf v_1\otimes \mathbf v_1$ gives vectors of the form $\mathbf v_a\otimes \mathbf v_b\otimes \mathbf v_1\otimes \mathbf v_1$. A further application the factor
\beq\label{pf3}
\Big(1-\frac{P_{13}'+P_{23}'}{\tl a+\ka_1}\Big)\Big(1-\frac{P_{14}'+P_{24}'}{\tl a}\Big)
\eeq
on $\mathbf v_a\otimes \mathbf v_b\otimes \mathbf v_1\otimes \mathbf v_1$ results in
\begin{align*}
&\mathbf v_a\otimes \mathbf v_b\otimes \mathbf v_1\otimes \mathbf v_1\\
&-\frac{1}{\tl a+\ka_1}(\mathbf v_a\otimes \mathbf v_{1'}\otimes \mathbf v_{b'}\otimes \mathbf v_{1}(-1)^{|1||b|}+\mathbf v_a\otimes \mathbf v_{1'}\otimes \mathbf v_1\otimes \mathbf v_{b'}(-1)^{|1|})\\
&+\frac{1}{\tl a(\tl a+\ka_1)}(\mathbf v_{1'}\otimes \mathbf v_{1'}\otimes \mathbf v_{a'}\otimes \mathbf v_{b'}+\mathbf v_{1'}\otimes \mathbf v_{1'}\otimes \mathbf v_{b'}\otimes \mathbf v_{a'}(-1)^{|1|+|a||b|}) \\
&-\frac{1}{\tl a+\ka_1}(\mathbf v_{1'}\otimes \mathbf v_b\otimes \mathbf v_1\otimes \mathbf v_{a'}(-1)^{|a||b|+|1||b|+|1|}+\mathbf v_{1'}\otimes \mathbf v_b\otimes \mathbf v_{a'}\otimes \mathbf v_{1}(-1)^{|a||b|+|1||a|+|1||b|}).
\end{align*}
Again by Lemma \ref{Blem} (3), $\Gamma_{34}(v)$ annihilates the above vectors. Thus it suffices to consider the  further application of the rest factors acting on the terms in the first line of \eqref{pf2}. Note that the action of $\Gamma_{12}(u)$ gives vectors of the form $\mathbf v_{a}\otimes \mathbf v_b\otimes \mathbf v_1\otimes \mathbf v_{c}$ where $2\lle c\lle 2'$. A similar calculation as above shows that the restriction of the image of $\mathbf v_{a}\otimes \mathbf v_b\otimes \mathbf v_1\otimes \mathbf v_{c}$ under the operator
$$
\Gamma_{34}(v)\Big(1-\frac{P_{13}'+P_{23}'}{\tl a+\ka_1}\Big)\Big(1-\frac{P_{14}'+P_{24}'}{\tl a}\Big)
$$
to the subspace spanned by the vectors of the form $\mathbf v_{1}\otimes \mathbf v_i\otimes \mathbf v_1\otimes \mathbf v_k$ with $2\lle i,k\lle 2'$ is nonzero only if $a=1$ and $2\lle b\lle 2'$. Moreover,
\begin{align*}
\Big(1-\frac{P_{13}'+P_{23}'}{\tl a+1}\Big)\Big(1-\frac{P_{14}'+P_{24}'}{\tl a}\Big)\mathbf v_{1}\otimes \mathbf v_b\otimes \mathbf v_1\otimes \mathbf v_{c}\equiv  \Big(1-\frac{P_{24}'}{\tl a}\Big)\mathbf v_{1}\otimes \mathbf v_b\otimes \mathbf v_1\otimes \mathbf v_{c},
\end{align*}
where the symbol $\equiv$ means we only keep the basis vectors which can give a nonzero contribution to the coefficients of $\mathbf v_1\otimes \mathbf v_i\otimes \mathbf v_1\otimes \mathbf v_k$ after the subsequent application of the operator $\Gamma_{34}(v)$.

To sum up, we have proved that the restriction of the operator on the right hand side of \eqref{pf1} to the subspace spanned by the basis vectors of the form $\mathbf v_1\otimes \mathbf v_j\otimes \mathbf v_1\otimes \mathbf v_l$ with $2\lle j,l\lle 2'$ coincides with the operator
\beq\label{pf4}
\begin{split}
\frac{a-2\ka_1}{a-\ka_1}\,\Gamma_{34}(v)&\Big(1-\frac{P_{24}'}{\tl a}\Big)\Gamma_{12}(u)\Big(1-\frac{P_{24}}{a}\Big)\\
&=\frac{a-2\ka_1}{a-\ka_1}\,\Gamma_{34}(v)R_{24}'(u+v)\Gamma_{12}(u)R_{24}(u-v).
\end{split}
\eeq
Here $R_{24}(u-v)$ and $R_{24}'(u+v)$ are the R-matrices used to define the extended twisted super Yangian $\scrX_{[2,2']}^\bs$. Moreover, the matrix elements for this restriction involve only the series $\Gamma_{1j}^{1i}(u)$ with $2\lle i,j\lle 2'$.

For the left hand side of \eqref{pf1}, we again apply it to the basis vectors of the form $\mathbf v_1\otimes \mathbf v_j\otimes \mathbf v_1\otimes \mathbf v_l$ with $2\lle j,l\lle 2'$ and look at the coefficients of the basis vectors of the same form in the image. Then the same argument as for the right hand side (with the reversed factors in the operators) implies that the coefficients of such basis vectors coincide with those of the operator
\beq\label{pf5}
\begin{split}
\frac{a-2\ka_1}{a-\ka_1}\,&\Big(1-\frac{P_{24}}{a}\Big)\Gamma_{12}(u)\Big(1-\frac{P_{24}'}{\tl a}\Big)\Gamma_{34}(v)\\&=\frac{a-2\ka_1}{a-\ka_1}\,R_{24}(u-v)\Gamma_{12}(u)R_{24}'(u+v)\Gamma_{34}(v).
\end{split}
\eeq
Again, $R_{24}(u-v)$ and $R_{24}'(u+v)$ are the R-matrices used to define the extended twisted super Yangian $\scrX_{[2,2']}^\bs$. Moreover, the matrix elements for this restriction involve only the series $\Gamma_{1j}^{1i}(u)$ with $2\lle i,j\lle 2'$.

Therefore, by equating the matrix elements of the operators \eqref{pf4} and \eqref{pf5}, we get the R-matrix form of the defining relations for the superalgebra $\scrX_{[2,2']}^\bs$ is satisfied by the series 
$$
\mathcal T_{ij}(u)=\Gamma_{1j}^{1i}(u)(-1)^{|i||j|+|j|},
$$ see Lemma \ref{Blem} (2), as required.
\end{proof}

\begin{rem}
In the orthosymplectic Yangian case \cite{Molev24drinfeld,frassek2023orthosymplectic}, the proof of \cite[Lem.~3.6]{Jing18iso} requires a specialization of the orthosymplectic R-matrix at a value where a denominator may vanish. More precisely, for the orthosymplectic R-matrix $R(u)=1-P/u+Q/(u-\kappa)$, the value $R(u)$ is regular at $u=\pm 1$ unless \(\kappa=1\). In our type A super setting, the Yang R-matrix \(R(u)=1-P/u\) is always regular at \(u=\pm1\). Therefore, the same strategy as in \cite[Lem.~3.6]{Jing18iso} applies here without this singularity issue.
\end{rem}

We also need the following generalization of Proposition \ref{embedB}. Fix a positive integer $m$ such that
$m\leqslant \ell$ if $N=2\ell+1$ and $m\leqslant \ell-1$ if $N=2\ell$.
Suppose that the generators $x_{ij}^{(r)}$ of the superalgebra $\scrX^\bs_{(m,m')}$ are labelled by the indices
$m+1\leqslant i,j\leqslant (m+1)'$ and $r>0$.

\begin{prop}\label{thm:red}
The mapping
\beq\label{redu}
\psi_m^\bs:x_{ij}(u)\mapsto \left|\begin{matrix}
x_{11}(u)&\dots&x_{1m}(u)&x_{1j}(u)\\
\dots&\dots&\dots&\dots\\
x_{m1}(u)&\dots&x_{mm}(u)&x_{mj}(u)\\
x_{i1}(u)&\dots&x_{im}(u)&\mybox{$x_{ij}(u)$}
\end{matrix}\right|,\qquad m+1\leqslant i,j\leqslant (m+1)',
\eeq
defines a superalgebra homomorphism $\scrX^\bs_{(m,m')}\to \scrX^\bs$.
\end{prop}
\begin{proof}
The proof is parallel to that of \cite[Proposition 3.7]{Jing18iso} by using the Sylvester theorem for quasi-determinants.
\end{proof}

The homomorphisms $\psi_m^\fks$ have the following consistence property. For $l\in\bN$, we have the corresponding homomorphism
\[
\psi_m^{\bs_{(l,l')}}:\scrX^\bs_{(m+l,(m+l)')}\to \scrX^\bs_{(l,l')}
\]
given by \eqref{redu}.

\begin{cor}\label{cons}
We have the equality of superalgebra homomorphisms,
\[
\psi^\bs_l\circ \psi_m^{\bs_{(l,l')}}=\psi^\bs_{m+l}.
\]
\end{cor}
\begin{proof}
Follows from the same argument as in \cite[Proposition 3.8]{Jing18iso}.
\end{proof}

\begin{cor}\label{cor:commu}
We have the relations
\be
\big[x_{ab}(u),\psi_m^\fks(x_{ij}(v))\big]=0
\ee
for all $1\leqslant a,b\leqslant m$ and $m+1\leqslant i,j\leqslant (m+1)'$. 
\end{cor}
\begin{proof}
Write $A(u)=\big(x_{ab}(u)\big)_{1\lle a,b\lle m}$. For \(m+1\lle i,j\lle (m+1)'\), the image \(\psi_m^\fks(x_{ij}(u))\) is the
\((i,j)\)-entry of the Schur complement of the upper-left \(m\times m\)
block:
\[
\psi_m^\fks(x_{ij}(u))
=
x_{ij}(u)-\sum_{c,d=1}^m x_{ic}(u)\,\big(A(u)^{-1}\big)_{cd}\,x_{dj}(u),
\]
equivalently the quasideterminant in \eqref{redu}.

Observe that, for $1\lle a,b,c,d\lle m$, $m+1\lle i,j\lle (m+1)'$,  the following commutation relations \eqref{bcom}
reduce exactly to the corresponding relations for the super Yangian
\(\rY(\gl^\bs)\):
\begin{align*}
(u-v)[x_{ab}(u),x_{ij}(v)]
&=
(-1)^{|a||b|+|a||i|+|b||i|}
\big(x_{ib}(u)x_{aj}(v)-x_{ib}(v)x_{aj}(u)\big),\\
(u-v)[x_{ab}(u),x_{ic}(v)]
&=
(-1)^{|a||b|+|a||i|+|b||i|}
\big(x_{ib}(u)x_{ac}(v)-x_{ib}(v)x_{ac}(u)\big),\\
(u-v)[x_{ab}(u),x_{cd}(v)]
&=
(-1)^{|a||b|+|a||c|+|b||c|}
\big(x_{cb}(u)x_{ad}(v)-x_{cb}(v)x_{ad}(u)\big),\\
(u-v)[x_{ab}(u),x_{dj}(v)]
&=
(-1)^{|a||b|+|a||d|+|b||d|}
\big(x_{db}(u)x_{aj}(v)-x_{db}(v)x_{aj}(u)\big).
\end{align*}
Thus we can express the LHS of the desired equality in terms of a linear combination of ordered PBW-type supermonomials with an order satisfying $x_{ic}^{(r_1)}<x_{ab}^{(r_2)}<x_{dj}^{(r_3)}$ for all $r_1,r_2,r_3\in \mathbb{Z}_{>0}$. Namely, we always put $x_{ic}^{(r_1)}$ to the LHS while put $x_{dj}^{(r_3)}$ to the RHS. The elements $x_{ab}^{(r_2)}$ are always in the middle. It is not hard to see from the classical picture \eqref{isogr} that such ordered PBW supermonomials are linearly independent in $\mathscr X^\fks$. Then to rewrite the LHS of the desired equality in terms of these ordered supermonomials, the choice of the order ensures that all the computations only use the simplified relations in the super Yangian
\(\rY(\gl^\bs)\).
Since it has been proved in \cite[Lemma 4.3]{Peng2016parabolic} that the desired equality holds in the super Yangian $\rY(\gl^\bs)$, we conclude it also holds in $\mathscr X^\fks$; see also \cite[Coro. 3.10]{Jing18iso}, \cite[Coro. 3.3]{Molev24drinfeld} and \cite[Coro. 3.52]{frassek2023orthosymplectic}.
\end{proof}


\subsection{Gaussian generators and their properties}
Recall the definition of the central series $c(u)$ from \eqref{cudef}.

\begin{lem}\label{efaiii}
In the superalgebra $\scrX^\bs$, we have $c(u)\tl d_{i'}(u)=d_{i}(-u)$,
\[
\tl d_i(u)d_{i+1}(u)=\tl d_{i'-1}(-u)d_{i'}(-u),\quad  \tl e_{ij}(u)=f_{i'j'}(-u), \quad  \tl f_{ji}(u)=e_{j'i'}(-u).
\]
In particular, we have $\wtl E(u)=F(-u)'$ in matrix form and $e_{i}(u)=-f_{\tau i}(-u)$.
\end{lem}
\begin{proof}
This is completely parallel to the proof of \cite[Lem. 6.8]{lu2024drinfeld}.
\end{proof}

\begin{lem}\label{shiftlem}
Suppose $m\lle \lfloor \tfrac{N-1}{2}\rfloor$, then the homomorphism $\psi^\bs_m:\scrX^\bs_{(m,m')}\to \scrX^\bs$ sends
\[
d_i(u)\to d_{m+i}(u),\quad e_{ij}(u)\to e_{m+i,m+j}(u),\quad f_{ji}(u)\to f_{m+j,m+i}(u).
\]
\end{lem}
\begin{proof}
The lemma follows from Propositions \ref{embedB}, \ref{cons}; cf. the proof of \cite[Corollary 3.2]{Lu2023drinfeld}.
\end{proof}

Due to Lemma \ref{shiftlem}, we call $\psi^\bs_m$ a \textit{rank-reduction homomorphism}.

\begin{lem}\label{lemnew1}
Suppose $m\lle \lfloor \tfrac{N-1}{2}\rfloor$. We have $[d_i(u),d_j(v)]=0$ for $1\lle i,j\lle N$ and 
\[
[d_{i}(u),e_j(v)]=[d_{i}(u),f_j(v)]=0,
\]
for (1) $1\lle i\lle m<j<(m+1)'$ and (2) $1\lle j< m<i\lle (m+1)'$.
\end{lem}
\begin{proof}
Note that $[d_1(u),d_1(v)]=0$ follows from \eqref{bcom} and this implies by Lemma \ref{efaiii} and Lemma \ref{shiftlem} that $[d_i(u),d_i(v)]=0$ for $1\lle i\lle N$. The other relations are corollaries of Corollary \ref{cor:commu}, Lemma \ref{efaiii}, and Lemma \ref{shiftlem}. 
\end{proof}

\begin{lem}\label{lem:eta}
There is an anti-automorphism $\eta$ for $\mathscr X^\bs$ (and for $\Y^\bs$) defined by
\begin{align}\label{eta}
\eta:X(u)\longrightarrow X^t(u),\qquad x_{ij}(u)\mapsto x_{ji}(u)(-1)^{|i||j|+|j|}.
\end{align}
Moreover, for $1\lle i<j\lle N$ and $\ 1\lle k\lle N$, we have
\[
\eta\big(e_{ij}(u)\big)=f_{ji}(u)(-1)^{|i||j|+|j|},\quad \eta\big(f_{ji}(u)\big)=e_{ij}(u)(-1)^{|i||j|+|i|},\quad \eta\big(d_{k}(u)\big)=d_{k}(u).
\]
\end{lem}
\begin{proof}
It is straightforward to prove that $\eta$ defines an anti-automorphism for $\mathscr X^\bs$ and $\Y^\bs$. Applying $\eta$ to \eqref{eq:sij-Gauss}, the second statement follows from the uniqueness of Gauss decomposition.
\end{proof}

\begin{lem}\label{lem:de-gen}
The superalgebra $\scrX^\bs$ is generated by the coefficients of $d_i(u)$ and $e_j(u)$, where $1\lle i\lle N$ and $1\lle j<N$.
\end{lem}
\begin{proof}
We say that a series in $u^{-1}$ can be generated if its coefficients can be generated by the coefficients of $d_i(u)$ and $e_j(u)$, where $1\lle i\lle N$ and $1\lle j<N$.

By Lemma \ref{efaiii} and \eqref{eq:def-tilde-e-f}, it suffices to show that $e_{kl}(u)$ with $k<l$ can be generated. We prove it by induction on $N$. The base case $N=2$ is trivial. Now assume $N\gge 3$. Then it follows from Lemma \ref{shiftlem} and the induction hypothesis that $d_{k}(u), e_{ij}(u), f_{ji}(u)$ for $2\lle k\lle 2'$ and $2\lle i<j\lle 2'$ can be generated.

We now prove that $e_{1j}(u)$ is generated for all $j\gge 3$. Note that $x_{1j}(u)=d_1(u)e_{1j}(u)$, it suffices to show $x_{1j}(u)$ can be generated by another induction on $j$. Now let $2\lle j<N$ and suppose that $x_{1k}(u)$ for $1\lle k\lle j$ can be generated. By \eqref{bcom}, we have
\begin{align*}
(u^2-v^2)[x_{1j}(u),x_{j,j+1}(v)]=&\,(u+v)\ka_j\big(x_{jj}(u)x_{1,j+1}(v)-x_{jj}(v)x_{1,j+1}(u)\big)\\
+& \, (u-v)\ka_j\Big(\delta_{jj'}\sum_{a=1}^N x_{1a'}(u)x_{a,j+1}(v)-\delta_{1(j+1)'}\sum_{a=1}^N x_{ja'}(v)x_{aj}(u)\Big).
\end{align*}
Taking the coefficients of $v$, we have
\beq\label{0000-}
x_{1,j+1}(u)=\ka_j[x_{1j}(u),x_{j,j+1}^{(1)}]-\delta_{jj'}x_{1,j'-1}(u)+\delta_{1,j'-1}x_{j'j}(u)
\eeq
Clearly, $[x_{1j}(u),x_{j,j+1}^{(1)}]$ can be generated. If $j=j'$, then $x_{1,j'-1}(u)=x_{1,j-1}(u)=d_1(u)e_{1,j-1}(u)$ can be generated by induction hypothesis. If $j'=2$, then $x_{j'j}(u)=x_{2j}(u)$,  expressed by
\[
x_{j'j}(u)=\begin{cases}d_2(u)e_{2j}(u)+f_{1}(u)d_1(u)e_{1j}(u),&\text{ if }j> 2,\\
d_2(u)+f_{1}(u)d_1(u)e_{1}(u),  &\text{ if }j=2,
\end{cases}
\]
can also be generated by induction hypothesis. Thus it follows from \eqref{0000-} that $x_{1,j+1}(u)$ can be generated, completing the proof that $e_{1j}(u)$ for $j\gge 3$ can be generated.

By a similar induction, one proves that $f_{j1}(u)$ can also be generated. Finally, using \eqref{eq:def-tilde-e-f}, we find that $e_{jN}(u)=\tl f_{j'1}(-u)$ can be generated. 

Combining the observations above, we find that $e_{kl}(u)$ with $1\lle k<l\lle N$ can be generated, and hence the proof is complete.
\end{proof}

\subsection{Special twisted super Yangians}\label{sec:hb}
Let $\ell=\lfloor\frac{N}{2}\rfloor$. For a fixed symmetric parity sequence $\bm\ka$ and $1\lle i\lle N$, we introduce
\begin{align}\label{eq:rho}
\varrho_i=\sum_{j=1}^i \ka_j,\qquad \varkappa=\frac12\varrho_N=\frac12\sum_{i=1}^N \ka_i=\frac12(\fkm-\fkn).
\end{align}
Since $\bm\ka$ is symmetric, we always have $\varrho_i+\varrho_{\tau i}=2\varkappa$. By convention $\varrho_0=0$.

Set $d_0(u)=d_{N+1}(u)=1$. We define the following generating series, for $0\lle i\lle N$ and $1\lle j< N$,
\begin{enumerate}
    \item if $N=2\ell$ is even, then we set
\begin{align}
& b_j(u)= \sqrt{-1}\,f_j(u+\tfrac{\varkappa-\varrho_j}{2}),\label{beven}\\
& h_i(u)=\tl d_{i}(u+\tfrac{\varkappa-\varrho_i}{2})d_{i+1}(u+\tfrac{\varkappa-\varrho_i}{2})\label{heven};
\end{align}
\item if $N=2\ell+1$ is odd, then we set
\begin{align}
&b_j(u)= \sqrt{-1}f_j(u+\tfrac{\varkappa-\varrho_j}{2}),\label{bodd}\\
&h_i(u)=
\begin{cases}
\tl d_{i}(u+\tfrac{\varkappa-\varrho_i}{2})d_{i+1}(u+\tfrac{\varkappa-\varrho_i}{2}), &\text{ if } i\ne \ell,\ell+1,\\
\big(1+\tfrac{\ka_{\ell+1}}{4u}\big)\tl d_{i}(u+\tfrac{\ka_{\ell+1}}{4})d_{i+1}(u+\tfrac{\ka_{\ell+1}}{4}),& \text{ if } i=\ell,\\
\big(1-\tfrac{\ka_{\ell+1}}{4u}\big)\tl d_{i}(u-\tfrac{\ka_{\ell+1}}{4})d_{i+1}(u-\tfrac{\ka_{\ell+1}}{4}),& \text{ if } i=\ell+1.
\end{cases}\label{hodd}
\end{align}
\end{enumerate}
\begin{rem}
Note that our special shifts satisfy
\begin{align*}
& \frac{\varkappa-\varrho_i}{2}+\frac{\varkappa-\varrho_{\tau i}}{2}=0.
\end{align*}
The factor $\sqrt{-1}$ is introduced only to adjust a sign, so that the resulting generators satisfy the relations in Definition \ref{deftY}. These relations specialize, when $\tau=\mathrm{id}$, to the Drinfeld presentation of split twisted Yangians \cite{Lu2023drinfeld,lu2025affine}. The other modifications in the definitions will be justified by the low-rank computations carried out later. In particular, when $N=2\ell+1$, the special treatment of the indices $\ell$ and $\ell+1$ is designed precisely so that the corresponding generators satisfy the relations in Definition \ref{deftY}.
\end{rem}
By \eqref{unimat}, \eqref{cudef} and Lemma \ref{efaiii}, we have the following relation in $\mathscr Y^{\fks}$:
\beq\label{hprod}
\prod_{i=0}^N h_{i}\big(u-\tfrac{\varkappa-\varrho_i}{2}\big)=1-\delta_{N,\mathrm{odd}}\frac{1}{16u^2}.
\eeq
It follows from Lemma \ref{efaiii} that
\begin{align}
\label{i=tauih}
&h_{\tau i}(u)=h_{i}(-u),\\  
&b_{i}(u)=-\sqrt{-1}\,e_{\tau i}(-u+\tfrac{\varkappa-\varrho_{\tau i}}{2}).\label{fi=tauie}
\end{align}
In particular, if $N=2\ell$ is even, then $h_\ell(u)$ is an even series in $u^{-1}$. Moreover, we have
\beq\label{etahb}
\eta(h_i(u))=h_i(u),\qquad \eta(b_{i}(u))=-b_{\tau i}(-u)(-1)^{|i||i+1|+|i|}.
\eeq

Introduce $h_{i,r}$ and $b_{j,r}$ for $0\lle i\lle N$, $1\lle j<N$, and $r\in\bN$ as follows,
\begin{align}
h_i(u)=1+\sum_{r\gge 0}h_{i,r}u^{-r-1},\qquad b_j(u)=\sum_{r\gge 0}b_{j,r}u^{-r-1},\label{bhcom}
\end{align}
namely they are coefficients of $h_i(u)$ and $b_j(u)$.
\begin{dfn}\label{defSY}
The \emph{special twisted super Yangian $\SY^\bs$} is the subalgebra of $\Y^\bs$ generated by $b_{i,r}$ and $h_{i,r}$ for $1\lle i<N$ and $r\in\bN$.    
\end{dfn}

Define the root vectors $b_{\alpha,r}=b_{ji;r}$ for $\alpha=\alpha_i+\cdots+\alpha_{j-1}$ with $1\lle i<j\lle N$ and $r\in\bN$ recursively as follows,  
\beq\label{bjir2}
b_{\alpha_i,r}=b_{i+1,i;r}=b_{i,r},\qquad b_{\alpha,r}=b_{ji;r}=[b_{j-1,0},b_{j-1,i;r}].
\eeq
Recall the sets from \eqref{Itau},
\beq\label{I+-A}
\begin{split}
&\I_{\ne}^0=\{1,\dots,\ell-1\},\qquad \I_{=}^0=\{\ell\},\qquad \text{ if }N=2\ell;\\
&\I_{\ne}^0=\{1,\dots,\ell\},\hskip 1.45cm \I_{=}^0=\varnothing,\hskip 1.cm \text{ if }N=2\ell+1.
\end{split}
\eeq
\begin{prop}\label{prop:pwb}
The ordered super monomials of 
\beq\label{dd34}
\{b_{\alpha,r},h_{0,2r},h_{i,r},h_{j,2r+1}~|~\alpha\in\cR^+,i\in\I^0_{\ne},j\in\I_{=}^0,r\in\N\}
\eeq
(with respect to any fixed total ordering) form a basis in  $\Y^\bs$. In particular, the ordered super monomials of 
\beq\label{dd341}
\{b_{\alpha,r},h_{i,r},h_{j,2r+1}~|~\alpha\in\cR^+,i\in\I_{\ne}^0,j\in\I_{=}^0,r\in\N\}
\eeq
(with respect to any fixed total ordering) are linearly independent in  $\SY^\bs$.
\end{prop}
\begin{proof}
Let $\bar b_{\alpha;r}$, $\bar h_{i,r}$ be the images of $b_{\alpha;r}$, $h_{i,r}$ in the associated graded superalgebra $\gr\Y^\bs$, respectively. Then by \eqref{isogr}, we have
\beq\label{helper1}
\begin{split}
&\bar b_{\alpha_i;r}=\sqrt{-1}\,\ka_{i+1}\big(e_{i+1,i}+(-1)^re_{i'-1,i'}\big)z^r,\\
&\bar h_{0,r}=\ka_1\big(e_{11}+(-1)^re_{NN}\big)z^r,\\
&\bar h_{i,r}=\big(\ka_{i+1}e_{i+1,i+1}-\ka_ie_{ii}+(-1)^r(\ka_{i+1}e_{i'-1,i'-1}-\ka_{i}e_{i'i'})\big)z^r.    
\end{split}
\eeq
If $1\lle i<j\lle N$, we set $\alpha=\alpha_i+\cdots+\alpha_{j-1}$. Then it follows from \eqref{bjir2} that
\beq\label{helper2}
\begin{split}
\bar b_{\alpha,r}\in \varsigma_{\alpha}\big(f_{\alpha}+(-1)^r\vartheta(f_{\alpha})\big)z^r &+\sum_{1\lle i<N} \bC\bar h_{i,r}\\
&+\sum_{\mu:\mathrm{ht}(\mu) <\mathrm{ht}(\alpha)}\C\big(f_{\mu}+(-1)^r\vartheta(f_{\mu})\big)z^r
,
\end{split}
\eeq
where $\varsigma_{\alpha}\in \C^\times$ and $\mathrm{ht}(\mu)$ denotes the height of the root $\mu$. It follows from \eqref{helper1}, \eqref{helper2} and the PBW theorem that the images of the ordered super monomials of these elements form a basis in the associated graded superalgebra $\gr\,\Y^\bs\cong \rU(\gl^\bs[z]^\vartheta)$, completing the proof.
\end{proof}

We will see soon that the ordered super monomials of 
\eqref{dd341} also form a basis for the corresponding superalgebra $\SY^\bs$. 

\section{Main results}\label{sec:mainres}

\subsection{Explicit isomorphism}\label{sec:iso-GD}
In this subsection we discuss the explicit isomorphism between the twisted super Yangian $\Yi$ defined in Drinfeld type presentation and the (special) twisted super Yangians constructed via R-matrix presentation.

\begin{thm}\label{main2}
There is a superalgebra isomorphism 
\beq\label{isoexp}
\begin{split}
\Phi: &\,\Yi\to \SY^\bs,\\
&\, h_{i,r}\mapsto h_{i,r},\quad b_{i,r}\mapsto b_{i,r},
\end{split}
\eeq
for $i\in \I^0$, $r\in\bN$.
\end{thm}
\begin{proof}
In the next section, we shall prove that the defining relations for $\Yi$ are satisfied by the generators $h_{i,r},b_{i,r}$ of $\scrX^\bs$ constructed by Gauss decomposition; see \eqref{beven}--\eqref{hodd} and \eqref{bhcom}. We first verify these relations for small rank cases and apply the rank-reduction homomorphism (see Proposition \ref{thm:red} and Lemma \ref{shiftlem}) to obtain the general case. Since $\SY^\bs$ is a subquotient of $\scrX^\bs$, these relations also hold in $\SY^\bs$, proving that $\Phi$ is a superalgebra homomorphism. 

By Definition \ref{defSY}, $\Phi$ is surjective. Therefore, it suffices to show that $\Phi$ is injective which reduces to prove that a spanning set of $\Yi$ is sent to a set of linearly independent vectors of $\SY^\bs$. 

By Proposition \ref{prop:span}, the ordered super monomials of the elements in the set  \eqref{eq:span} with the index sets given by \eqref{Itau} in $\Yi$ form a spanning set of $\Yi$. It is clear from \eqref{rv}  and \eqref{bjir2} that $\Phi$ sends $b_{\alpha,r}$ to $b_{\alpha,r}$ for $\alpha\in\mc R^+$ and $r\in\bN$. Thus, these ordered super monomials are sent via $\Phi$ to the ordered super monomials of the elements in the set \eqref{dd34} in $\SY^\bs$ which are linearly independent in $\SY^\bs$ by Proposition \ref{prop:pwb}.
\end{proof}
In the course of proving Theorem \ref{main2}, we also obtain the following, cf. Proposition \ref{prop:span} and Proposition \ref{prop:pwb}.
\begin{cor}\label{thm:pbw}
The ordered super monomials of
\beq\label{eq:span2}
\big\{b_{\alpha,r},h_{i,r},h_{j,2r+1}\mid\alpha\in\mc R^+,i\in \I^0_{\ne}, j\in \I^0_{=},r\in\bN\big\}
\eeq
(with respect to any fixed total ordering) form a basis of $\Yi$ (resp. $\SY^\bs$). In particular, we have $$\gr\,\SY^\bs\cong \rU(\mathfrak{sl}^\bs[z]^\vartheta).$$
\end{cor}   

We can also obtain a Drinfeld type presentation for the twisted super Yangian $\Y^\bs$.  Note that $h_{0,2r}$ for $r\in\bN$
are additional generators in $\Y^\bs$, which are not present in $\SY^\bs$. To that end, we use the following notation.  

Set $\tau(0)=N$, $\tau(N)=0$, $c_{0i}=-\ka_1\delta_{1i}$ and $c_{Ni}=-\ka_N\delta_{N-1,i}$.
\begin{thm}\label{main1}
The twisted super Yangian $\Y^\bs$ is isomorphic to the unital superalgebra generated by $h_{i,r}$, $b_{j,r}$, $0\lle i\lle N$, $1\lle j<N$, $r\in\bN$, where $|h_{i,r}|=\bar 0$ is even and $b_{j,r}$ is of parity $|\alpha_j|$, subject to the relations \eqref{qsconj0}--\eqref{qsconj10}, \eqref{hprod}, and $h_0(u)=h_N(-u)$. 
\end{thm}
\begin{proof}
Let $\mathfrak Y^\bs_\imath$ be the superalgebra generated by $h_{i,r},b_{j,r}$ for $0\lle i\lle N$, $1\lle j<N$, $r\in\bN$ subject to the relations \eqref{qsconj0}--\eqref{qsconj10}, \eqref{hprod}, and $h_0(u)=h_N(-u)$. Similar to the proof of Theorem \ref{main2}, there is a superalgebra homomorphism
\[
\Xi:\mathfrak Y_\imath^\bs\to \Y^\bs,\quad h_{i,r}\mapsto h_{i,r},\quad b_{j,r}\mapsto b_{j,r}.
\]
Indeed, in the next section, we prove that the relations \eqref{qsconj0}--\eqref{qsconj10} are satisfied in the extended twisted super Yangian $\scrX^\bs$ and hence in $\Y^\bs$. The relation \eqref{hprod} holds by its construction. By Proposition \ref{Binv}, Lemma \ref{efaiii}, and the definition of $h_i(u)$ in \eqref{heven} and \eqref{hodd}, the unitary relation \eqref{bunit} or \eqref{unimat} is equivalent to $d_1(u)=\tl d_N(-u)$, i.e. $h_0(u)=h_{N}(-u)$.

It remains to prove that $\Xi$ is an isomorphism. By Lemma \ref{lem:de-gen} and \eqref{fi=tauie}, $\Xi$ is surjective. Then we prove the injectivity. By the relation \eqref{hprod} and $h_0(u)=h_N(-u)$, we conclude that
\[
h_0\big(u-\tfrac{\varkappa}{2}\big)h_0\big(\hskip -0.1cm -u-\tfrac{\varkappa}{2}\big)=\Big(1-\delta_{N,\mathrm{odd}}\frac{1}{16u^2}\Big)\prod_{1\lle i<N}h_{i}\big(u-\tfrac{\varkappa-\varrho_i}{2}\big)^{-1},
\]
where $\varrho_i$ and $\varkappa$ are defined in \eqref{eq:rho}. Therefore, $h_{0,2s+1}$ for $s\in\bN$ can be expressed as polynomials in $h_{0,2r}$ and $h_{i,r}$ for $1\lle i<N$ and $r\in\bN$. Thus, arguing as in Proposition \ref{prop:span} with the help of Lemma \ref{lem:sym-new}, we prove that the ordered super monomials in the elements of $\{b_{\alpha,r},h_{0,2r},h_{i,r},h_{j,2r+1}~|~\alpha\in\cR^+,i\in\I^0_{\ne},j\in\I_{=}^0,r\in\N\}$ span the superalgebra $\mathfrak Y^\bs_\imath$, where $b_{\alpha,r}$ is defined the same way as in \eqref{rv}. By Proposition \ref{prop:pwb}, the images of these ordered super monomials under $\Xi$ are linearly independent and hence these ordered super monomials form a basis of $\mathfrak Y^\bs_\imath$, establishing the injectivity of $\Xi$. 
\end{proof}

\subsection{Center of twisted super Yangians}
As an application, we take the chance to discuss a set of algebraically independent generators of the center $\Y^\bs$ in terms of the generating series $d_i(u)$ for $i\in\I$. 

For a parity sequence, we introduce the following numbers $\gamma_{i}$, $i\in \I$, by the rule:
\beq\label{eq:gammadef}
\gamma_1=\varkappa-\tfrac12\ka_1
\qquad 
\gamma_{i+1}=\gamma_{i}-\tfrac{1}{2}(\ka_i+\ka_{i+1}).
\eeq
Note that we always have
\beq\label{gam}
\gamma_i+\gamma_{i'}=0,\qquad i\in\I.
\eeq

Define the \emph{quantum Berezinian of the matrix $X(u)$} by
\beq\label{cu}
\mathpzc{Ber}^\bs(u)=\prod_{i=1}^N d_i(u+\gamma_i)^{\ka_i}.
\eeq
For instance, if $N=2$ and $N=3$, then the corresponding $\mathpzc{Ber}(u)$ are, respectively, given by
\[
d_1(u+\tfrac{\ka_1+\ka_2}{4})^{\ka_1}d_2(u-\tfrac{\ka_1+\ka_2}{4})^{\ka_2},\quad d_1(u+\tfrac{\ka_1+\ka_2}{2})^{\ka_1}d_2(u)^{\ka_2}d_3(u-\tfrac{\ka_1+\ka_2}{2})^{\ka_3}.
\]

Note that by Lemma \ref{efaiii} and \eqref{gam} we have
\beq\label{symcu}
d_{N}(u+\gamma_N)^{\ka_N}\cdots d_\ell(u+\gamma_{N+1-\ell})^{\ka_{N+1-\ell}}=\prod_{i=1}^\ell d_{i'}(u+\gamma_{i'})^{\ka_{i'}}=\Big(\prod_{i=1}^\ell d_{i}(-u+\gamma_{i})^{\ka_{i}}\Big)^{-1}.
\eeq
Set 
\be 
\mathfrak C^\bs(u)=\prod_{i=1}^\ell d_i(u+\gamma_i)^{\ka_i}.
\ee 
It follows from \eqref{symcu} that
\beq\label{B=CC}
\mathpzc{Ber}^\bs(u)=
\begin{cases}\mathfrak C^\bs(u)\mathfrak C^\bs(-u)^{-1}, &\text{ if }N=2\ell,\\
\mathfrak C^\bs(u)d_{\ell+1}(u)^{\ka_{\ell+1}}\mathfrak C^\bs(-u)^{-1}, &\text{ if }N=2\ell+1.
\end{cases}
\eeq

Recall $c(u)$ from \eqref{cudef} and note that $c(u)=1$ in $\Y^\bs$; see \eqref{unimat} and Proposition \ref{Binv}. Then it follows from Lemma \ref{efaiii} that
\beq\label{cu-sym}
\mathpzc{Ber}^\bs(u)\mathpzc{Ber}^\bs(-u)=1.
\eeq
Define the elements $\mathcal C_r\in\Y^\bs$ by
\[
\mathpzc{Ber}^\bs(u)=1+\sum_{r\gge 1}\mathcal C_ru^{-r}.
\]
Denote by $\mathscr{ZY}^\bs$ the center of the twisted super Yangian $\Y^\bs$. 

In the rest of this subsection, we describe the center $\mathscr{ZY}^\bs$, and make connections with $\Y^\bs$ and $\SY^\bs$ using this center. Similar results to these features in the literature on extended Yangians and twisted Yangians can be found, for instance, in \cite[\S 2 and \S 4]{Molev96Yangians}, \cite{Guay2016twisted}, and \cite[\S4.2]{lu2024isomorphism}.

\begin{thm}\label{thm:center}
We have the following statements.
\begin{enumerate}
    \item The coefficients $\mathcal C_r$ of the series $\mathpzc{Ber}^\bs(u)$ are central in $\Y^\bs$.
    \item The elements $\mathcal C_{2r+1}$ for $r\in\bN$ are algebraically free generators of the center $\mathscr{ZY}^\bs$ of $\Y^\bs$.
    \item We have $\SY^\bs=\rY(\mathfrak{sl}^\bs)\cap \Y^\bs$.
\end{enumerate}
\end{thm}
\begin{proof}
(1) By Lemma \ref{lemnew1}, we have $[d_i(u),d_j(v)]=0$ for $1\lle i,j\lle N$. Thus it suffices, by Lemma  \ref{efaiii} and Lemma \ref{lem:de-gen}, to verify that 
$$
[\mathpzc{Ber}^\bs(u),e_i(v)]=[\mathpzc{Ber}^\bs(u),f_i(v)]=0,\qquad 1\lle i\lle \ell:=\lfloor \tfrac{N}{2}\rfloor.
$$ 
Recall the definition of $\rY(\gl_{[\ell]}^{\bm\ka})$ from \eqref{ysbs-sub}. By \eqref{bcom}, there is a homomorphism from
\[
\rY(\gl_{[\ell]}^{\bm\ka})\to \Y^\bs,\quad t_{ij}(u)\mapsto x_{ij}(u),\quad 1\lle i,j\lle \ell,
\]
see also Proposition \ref{propA} below. Thus if $i<\ell$, then by \cite[Thm.~2.43]{tsymbaliuk2020shuffle} we have $\mathfrak C^\bs(u)$ commutes with $e_i(v)$ and $f_i(v)$ if $1\lle i<\ell$. 
Hence it follows from Lemma \ref{lemnew1} and \eqref{B=CC} that $[\mathpzc{Ber}^\bs(u),e_i(v)]=[\mathpzc{Ber}^\bs(u),f_i(v)]=0$ for $1\lle i<\ell$. 

It remains to verify that $[\mathpzc{Ber}^\bs(u),e_\ell(v)]=[\mathpzc{Ber}^\bs(u),f_\ell(v)]=0$. Again by Lemma \ref{lemnew1}, it reduces to verify the statement for the case $N=2,3$ which will be done in Lemma \ref{centralb2} and Lemma  \ref{centralb3} below.

(2) It follows from \eqref{cu-sym} that $\mathcal C_{2r}$ can be expressed by $\mathcal C_i$ for $i<2r$. Thus it is easy to see by induction that all $\mathcal C_i$ can be expressed in terms of $\mathcal C_{2r+1}$ for $r\in\bN$. It suffices to prove the statement in the associated graded superalgebra $\gr\,\Y^\bs$. Let $\overline{\mc C}_i$ be the image of $\mc C_i$ in the $(i-1)$-st component of $\gr\,\Y^\bs$.

Recall from Section \ref{sec:basics} that $\gr\,\Y^\bs\cong \mathrm{U}(\gl^\bs[z]^\vartheta)$. By \eqref{isogr}, one easily sees that
\[
\overline{\mc C}_{2r+1}=2\mathcal Iz^{2r},\qquad \mathcal I=e_{11}+\cdots+e_{NN}.
\]
To complete the proof, it suffices to show that the center of the superalgebra $\mathrm{U}(\gl^\bs[z]^\vartheta)$ is generated by the elements $\mathcal I z^{2r}$ with $r\in\bN$. 
The rest is very similar to \cite[Prop.~4.10]{Molev96Yangians} and \cite[Lem.~6]{Gow2007gauss}. We sketch the proof. 

Instead of working on the twisted current superalgebra $\mathrm{U}(\gl^\bs[z]^\vartheta)$, we consider another twisted current superalgebra $\mathrm{U}(\gl^\bs[z]^\varpi)$ defined as follows, see e.g. \cite{Molev2002reflection,lu2023twisted}. Let $\bm\iota=(\iota_1,\cdots,\iota_N)$ be a sequence such that $\iota_i=1$ for $1\lle i\lle \ell$ and $\iota_i=-1$ otherwise. Let $\varpi$ be the involution of $\gl^\bs$ defined by
\[
\varpi:\gl^\bs\to\gl^\bs,\quad e_{ij}\mapsto \iota_i\iota_je_{ij}.
\]
Let $\gl^\bs_{\bar 0}$ be the fixed point subalgebra under $\varpi$ and $\gl^\bs_{\bar 1}$ the eigenspace of $\varpi$ associated to the eigenvalue $-1$,
$$
\gl^\bs_{\bar 0}=\bC\langle e_{ij}:\iota_i=\iota_j,i,j\in\I\rangle\cong \gl^{\bs_{[1,\ell]}}\oplus \gl^{\bs_{(\ell,N]}},\quad 
\gl^\bs_{\bar 1}=\bC\langle e_{ij}: \iota_i\ne \iota_j,i,j\in\I\rangle.$$ 
Then set
\[
\gl^\bs[z]^\varpi=\gl^\bs_{\bar 0}\oplus \gl^\bs_{\bar 1}z\oplus \gl^\bs_{\bar 0}z^2\oplus \gl^\bs_{\bar 1}z^3\oplus\cdots.
\]
The superalgebras $\gl^\bs[z]^\varpi$ and $\gl^\bs[z]^\vartheta$ are isomorphic by a conjugation. Under this isomorphism, $\mc I z^{2r}$ in $\gl^\bs[z]^\varpi$ corresponds to $\mc I z^{2r}$ in $\gl^\bs[z]^\vartheta$. Let $\mathrm S(\gl^\bs[z]^\varpi)$ denote the supersymmetric algebra of $\gl^\bs[z]^\varpi$.  It suffices to prove the corresponding result for $\mathrm S(\gl^\bs[z]^\varpi)$.

If $N=2$, then by our assumption, $\ka_1=\ka_2$. Thus all elements in the superalgebra $\Y^\bs$ are even. Hence this follows from the corresponding result for nonsuper case; see \cite[Thm. 3.4]{Molev2002reflection} and \cite[Thm. 6.19]{lu2024drinfeld}.

If $N\gge 3$, then it is easy to see that the $\gl^\bs_{\bar 0}$-module $\gl^\bs_{\bar 1}$ has no invariant elements. Thus using the same arguments in \cite[Prop.~2.12 \& Prop.~4.10]{Molev96Yangians}, one proves that the center is contained in the subalgebra $\mathrm S(\gl^\bs_{\bar 0}[z^2])$, where $\mathrm S(\gl_{\bar 0}^\bs[z^2])$ denote the supersymmetric algebra of $\gl_{\bar 0}^\bs[z^2]$. Since $\gl^\bs_{\bar 0}\cong \gl^{\bs_{[1,\ell]}}\oplus \gl^{\bs_{(\ell,N]}}$, by \cite[Lem.~6]{Gow2007gauss} for $\gl^\bs_{\bar 0}[z]$, the center is contained in the subalgebra generated by $\sum_{i=1}^\ell e_{ii}z^{2r}$ and $\sum_{i=\ell+1}^N e_{ii}z^{2r}$ for $r\in\bN$. Note that the center supercommutes with $\gl_{\bar 1}^{\bs}z$. Then again by the argument in \cite[Prop.~2.12]{Molev96Yangians}, one finds that the center is further contained in the subalgebra generated by $\mc Iz^{2r}$ for $r\in\bN$, completing the proof.

(3) It is clear by the definition of $\rY(\mathfrak{sl}^\bs)$ and the Gauss decomposition that $\SY^\bs\subset \rY(\mathfrak{sl}^\bs)\cap \Y^\bs$. Thus, to show the equality, it suffices to note that we have the corresponding equality in $\gr\,\rY(\gl^\bs)$ (recall that the filtration on $\Y^\bs$ is induced from the one on $\rY(\gl^\bs)$).
\end{proof}

Recall that $\mathfrak m$ is the number of $1$'s in the parity sequence $\bs$ while $\mathfrak n=N-\mathfrak m$.

\begin{prop}\label{prop:5.5}
We have the following statements.
\begin{enumerate}
\item If $\mathfrak m\ne \mathfrak n$, then we have a superalgebra isomorphism $\Y^\bs\cong \mathscr{ZY}^\bs\otimes \SY^\bs$. 
\item If $\mathfrak m= \mathfrak n$, then $\mathscr{ZY}^\bs\subset\SY^\bs$. 
\end{enumerate}
\end{prop}
\begin{proof}
(1) As in, e.g., \cite[Lem.~3.11]{Lu2023drinfeld}, one verifies that
$\Y^\bs=\mathscr{ZY}^\bs\cdot \SY^\bs$, where the condition
$\mathfrak m\ne \mathfrak n$ is needed. Thus it suffices to show that
$\mathscr{ZY}^\bs\cap\SY^\bs=\{1\}$. Again, this reduces to the associated
graded level. On the associated graded side, we use the standard decomposition
of the fixed-point current Lie superalgebra into its central trace part and
traceless part:
\[
\gl^\bs[z]^\vartheta
=
\bigoplus_{r\gge 0}\bC\,\mathcal I z^{2r}
\oplus
\mathfrak{sl}^\bs[z]^\vartheta,
\]
where the first summand is central. Hence, by the PBW theorem, we have the
vector space decomposition
\beq\label{grdec}
\mathrm{U}(\gl^\bs[z]^\vartheta)
=
\bC[\mathcal I z^{2r}]_{r\gge 0}\otimes
\mathrm{U}(\mathfrak{sl}^\bs[z]^\vartheta).
\eeq
Since the image of $\SY^\bs$ in $\gr\,\Y^\bs$ equals
$\mathrm{U}(\mathfrak{sl}^\bs[z]^\vartheta)$, \eqref{grdec} implies that the
associated graded images of $\mathscr{ZY}^\bs$ and $\SY^\bs$ intersect
trivially. Therefore $\mathscr{ZY}^\bs\cap\SY^\bs=\{1\}$.

(2) Note that $\bs$ is symmetric. If $\mathfrak m=\mathfrak n$, then $\mathfrak m=\mathfrak n$ is even and hence $N=2\ell$ with $\ell$ even. Moreover, the number of $1$'s in $\ka_{[1,\ell]}$ is the same as the number of $-1$'s. Recall $d_i(u)=\tl d_{i'}(-u)$ from Lemma \ref{efaiii} and $\gamma_i=\gamma_{i'}$ from \eqref{gam}. Now it suffices to show that $\mathfrak C^\bs(u)$ can be written as a product of $h_j(v)$ and $h_j(v)^{-1}$ with various $v$. This is essentially the same as \cite[Thm. 2.48(b)]{tsymbaliuk2020shuffle}.
\end{proof}

\begin{rem}
The use of \eqref{grdec} in the argument above is confined to Proposition~\ref{prop:5.5}. In particular, the earlier identification $\SY^\bs=\rY(\mathfrak{sl}^\bs)\cap \Y^\bs$ in Theorem~\ref{thm:center}(3) follows from a separate associated graded argument and is independent of \eqref{grdec}.
\end{rem}

\begin{rem}
One advantage of the Gaussian generators is that they provide direct access to the center. Indeed, an R-matrix description of the center via a Sklyanin-type superdeterminant is not presently available in the literature, while the Gauss decomposition naturally leads to the central series \(\mathpzc{Ber}^\bs(u)\) constructed above. It would be interesting to construct the super analogues of Sklyanin determinant of $S(u)$ and compare it with the central series $\mathpzc{Ber}^\bs(u)$; cf. \cite[Thm. 3.4]{Molev2002reflection}.
\end{rem}

\subsection{Isomorphism between different parity sequences} 
In this subsection, we discuss the relations between twisted super Yangians associated to different symmetric parity sequences.

Let $\bs,\tl{\bs}$ be parity sequences in $S_{\mathfrak m|\mathfrak n}$, then it is well known that $\rY(\gl^\bs)\cong \rY(\gl^{\tl{\bs}})$; see e.g. \cite{Huang19Solution,tsymbaliuk2020shuffle}. Specifically, take any $\sigma$ in the symmetric group $\mathfrak S_{N}$ such that $\ka_{i}=\tl{\ka}_{\sigma(i)}$ for $i\in\I$, then the map
\beq\label{eqPsigma}
\mathfrak P^\bs_{\sigma}:\rY(\gl^\bs)\to \rY(\gl^{\tl{\bs}}),\quad t_{ij}(u)\mapsto t_{\sigma(i)\sigma(j)}(u)
\eeq
defines a superalgebra isomorphism. Recall $\mathcal M_{g(u)}^\bs$ from \eqref{eq:mu_f-A}. Clearly, we have 
$$
\mathfrak P^\bs_{\sigma}\circ\mathcal M_{g(u)}^\bs=\mathcal M_{g(u)}^{\tl{\bs}}\circ \mathfrak P^\bs_{\sigma}.
$$
Thus we further have $\rY(\mathfrak{sl}^\bs)\cong \rY(\mathfrak{sl}^{\tl{\bs}})$.

Let $\bs,\tl{\bs}$ be symmetric in $S_{\mathfrak m|\mathfrak n}$, then $\Y^\bs\cong \Y^{\tl{\bs}}$ and $\scrX^\bs\cong \scrX^{\tl{\bs}}$. Specifically, take any $\sigma$ in the symmetric group $\mathfrak S_{N}$ such that $\ka_{i}=\tl{\ka}_{\sigma(i)}$ and $\sigma(i)'=\sigma(i')$ for $i\in\I$, then the map
\[
\mathfrak Q^\bs_{\sigma}:\scrX^\bs\to \scrX^{\tl{\bs}},\quad x_{ij}(u)\mapsto x_{\sigma(i)\sigma(j)}(u)
\]
defines a superalgebra isomorphism. By Proposition \ref{Binv}, $\mathfrak Q_{\sigma}(c^\bs(u))=c^{\tl{\bs}}(u)$ and hence we obtain the superalgebra isomorphism
\beq\label{eqPsigma2}
\mathfrak P^\bs_{\sigma}:\Y^\bs\to \Y^{\tl{\bs}},\quad x_{ij}(u)\mapsto x_{\sigma(i)\sigma(j)}(u).
\eeq
Here we use the same symbol $\mathfrak P^\bs_{\sigma}$ as the isomorphism \eqref{eqPsigma2} is compatible with \eqref{eqPsigma} by restriction if we regard $\Y^\bs$ and $\Y^{\tl{\bs}}$ as subalgebras of $\rY(\gl^\bs)$ and $\rY(\gl^{\tl{\bs}})$ via \eqref{inc}, respectively. Note that it is not hard to see from e.g. \cite[Proposition 3.5]{lu2023twisted} that $\mathfrak P^\bs_{\sigma}$ also preserves the coideal structures.

\begin{thm}\label{main3}
Let $\bs,\tl{\bs}$ be symmetric parity sequences in $S_{\mathfrak m|\mathfrak n}$, then $\mathbf Y^{\bs}_\imath \cong \mathbf Y^{\tl{\bs}}_\imath$, i.e. the twisted super Yangians are independent of the choice of symmetric parity sequences in $S_{\mathfrak m|\mathfrak n}$ (does depend on $\mathfrak m$ and $\mathfrak n$).
\end{thm}
\begin{proof}
Note that the superalgebras $\Y^\bs$ (resp. $\rY(\mathfrak{sl}^\bs)$ ) and $\Y^{\tl{\bs}}$ (resp. $\rY(\mathfrak{sl}^{\tl{\bs}})$) are isomorphic via the restriction of $\mathfrak P^\bs_\sigma$ in \eqref{eqPsigma}. Then the statement follows from Theorem \ref{main2} and Theorem \ref{thm:center} (3).
\end{proof}

\subsection{Quantum Berezinians in different parity sequences}
In this subsection, we discuss the relations between quantum Berezinians in different parity sequences; see e.g. \cite{Huang20duality,Chang23center} for the results of more general Manin (super)matrices or super Yangians of type A (cf. also \cite{tsymbaliuk2020shuffle,Lu2022note}). We first recall \cite[Prop. 3.6]{Huang20duality} and show by example how it implies \cite[Thm. 4.5]{Huang20duality}.

Let $\mc A$ be a superalgebra. We call the operators of the form
\[
\mc K =
\sum_{i,j\in\I} K_{ij}\otimes E_{ij}(-1)^{|i||j|+|j|}\in \mc A \otimes \End(\C^\bs),
\]
\emph{a matrix of parity sequence $\bs$} if $K_{ij}$ are elements of $\mc A$ of parity $|i|+|j|$ (determined by $\bs$). We simply write it as $\mc K=(K_{ij})_{i,j\in\I}$.

We say that $\mc K$ is a \emph{Manin matrix of parity sequence $\bs$} if $\mc K$ is of parity sequence $\bs$ and
\[
[K_{ij},K_{kl}]=(-1)^{|i||j|+|i||k|+|j||k|}[K_{kj},K_{il}]
\]
for all $i,j,k,l\in \I$, cf. \eqref{Trel}.

The symmetric group $\mathfrak S_N$ acts on matrices and parity sequences by the following rule. For $\sigma\in\mathfrak S_N$, we set $\sigma(\mc K)=(K_{\sigma^{-1}(i),\sigma^{-1}(j)})_{i,j\in\I}$ and $\sigma(\bs)=(\ka_{\sigma^{-1}(1)},\cdots,\ka_{\sigma^{-1}(N)})$. By \cite[Lem. 3.3]{Huang20duality}, if $\mc K$ is a Manin matrix of parity sequence $\bs$, then $\sigma(\mc K)$ is a Manin matrix of parity sequence $\sigma(\bs)$.

Suppose that $\mc K$ is a Manin matrix of parity sequence $\bs$ and has a Gauss decomposition (see Section \ref{sec:GD}) with the entries of the diagonal matrix given by $\mc D = \mathrm{diag}(\mc D_1,\cdots,\mc D_N)$. Here and below, we shall always assume that $\mc D_i$ are invertible for any choice of $\bs$. 

Define the \emph{Berezinian of $\mc K$ associated to the parity sequence $\bs$} by
\[
\mathpzc{Ber}^\bs(\mc K)=\mc D_1^{\ka_1}\mc D_2^{\ka_2}\cdots \mc D_N^{\ka_N}.
\]
The following shows that the action of the symmetric group $\mathfrak S_N$ does not change the Berezinian.

\begin{prop}[{\cite[Prop 3.6]{Huang20duality}}]\label{prophm}
Let $\mc K$ be a Manin matrix of parity sequence $\bs$ and $\sigma\in\mathfrak S_N$. Then
\[
\mathpzc{Ber}^\bs(\mc K)=\mathpzc{Ber}^{\sigma(\bs)}(\sigma(\mc K)).
\]
\end{prop}

It is well known that $T^\bs(u)e^{-\pa_u}$ is a Manin matrix of parity sequence $\bs$, where $T^\bs(u)$ is the generating matrix of the super Yangian $\rY(\gl^\bs)$ and $e^{-\pa_u}$ is the difference operator, i.e. $e^{-\pa_u}f(u)=f(u-1)$ for any function $f(u)$ in $u$; see e.g. \cite{Molev14MacMahon}. Suppose the diagonal matrix in the Gauss decomposition of $T^\bs(u)$ is given by $\mathfrak D^\bs(u)=\mathrm{diag}(\mathfrak D^\bs_1(u),\cdots,\mathfrak D^\bs_N(u))$. Then the diagonal matrix in the Gauss decomposition of $T^\bs(u)e^{-\pa_u}$ is given by 
$$
\mathfrak D^\bs(u)e^{-\pa_u}=\mathrm{diag}(\mathfrak D^\bs_1(u)e^{-\pa_u},\cdots,\mathfrak D^\bs_N(u)e^{-\pa_u}).
$$

Let us consider the following example which was used in \cite[\S4.3]{Lu21gl11XXX} and \cite[\S3.4]{Lu2021gelfand}.
\begin{eg}
Let $\mathfrak m=\mathfrak n=1$. Set $\bs=(1,-1)$ and $\tl{\bs}=(-1,1)$. Let $\sigma=(1,2)$ be the simple permutation. The matrix $\mc K$ is a Manin matrix of parity sequence $\bs$ if and only if
$$
[K_{11},K_{21}]=[K_{22},K_{21}]=[K_{21},K_{21}]=0,\qquad [K_{11},K_{22}]=[K_{12},K_{21}].
$$
Then the Berezinians of $\mc K$ associated to parity sequence $\bs$ and $\tl\bs$ are given by 
\begin{align*}
&\mathpzc{Ber}^\bs(\mc K)=K_{11}(K_{22}-K_{21}K_{11}^{-1}K_{12})^{-1},\\
&\mathpzc{Ber}^{\tl{\bs}}(\sigma(\mc K))=K_{22}^{-1}(K_{11}-K_{12}K_{22}^{-1}K_{21}).
\end{align*}
It is straightforward to check that $\mathpzc{Ber}^\bs(\mc K)=\mathpzc{Ber}^{\tl{\bs}}(\sigma(\mc K))$.

If $\mc K=T(u)^\bs e^{-\pa_u}$, then $\mathfrak D_1^\bs(u)=t_{11}^\bs(u)$ and $\mathfrak D_2^\bs(u)=t_{22}^\bs(u)-t_{21}^\bs(u)t_{11}^\bs(u)^{-1}t_{12}^\bs(u)$ while $\sigma(\mc K)=T^{\tl{\bs}}(u)e^{-\pa_u}$, $\mathfrak D_1^{\tl{\bs}}(u)=t_{11}^{\tl{\bs}}(u)=t_{22}^\bs(u)$ and $\mathfrak D_2^{\tl{\bs}}(u)=t_{11}^\bs(u)-t_{12}^\bs(u)t_{22}^\bs(u)^{-1}t_{21}^\bs(u)$. In terms of Gaussian generators, we have
\beq\label{99+1}
\mathpzc{Ber}^\bs(\mc K)= \mathfrak D_1^\bs(u)e^{-\pa_u}\big(\mathfrak D_2^\bs(u)e^{-\pa_u}\big)^{-1}=\big(\mathfrak D_1^{\tl{\bs}}(u)e^{-\pa_u}\big)^{-1}\mathfrak D_2^{\tl{\bs}}(u)e^{-\pa_u},
\eeq
which implies further
\beq\label{99+2}
\mathfrak D_1^\bs(u) \mathfrak D_2^\bs(u)^{-1}= \mathfrak D_1^{\tl{\bs}}(u+1)^{-1}\mathfrak D_2^{\tl{\bs}}(u+1).
\eeq
This shows the equality in \cite[Rem. 3.11]{Lu2021gelfand} and \cite[Ex. 3.2]{Lu2022note}; see also \cite[Lem. 4.3]{Chang23center}.
\end{eg}

The general case is no different and one obtains immediately the following. Let $\wp^\bs_1=0$ if $\ka_1=1$ and $\wp^\bs_1=1$ if $\ka_1=-1$.\footnote{The definition is clear from \eqref{99+1}--\eqref{99+2} as if $\ka_1=-1$ one has to move $e^{\pa_u}$ through to the right which creates a shift by $1$.} Define $\wp^\bs_i$ recursively for $2\lle i\lle N$ by $\wp^\bs_{i+1}=\wp^\bs_i-\frac{1}{2}(\ka_i+\ka_{i+1})$,
cf. \eqref{eq:gammadef}. 

The following is a corollary of Proposition \ref{prophm} with Proposition \ref{prophm2} which recovers \cite[Thm.~4.5]{Chang23center} (with all difference operators moved to the right and then dropped). It was previously used in \cite[Lem.~2.2]{Lu2021gelfand} and \cite[Ex.~3.3]{Lu2022note}. Recall $\mathfrak P_\sigma^\bs$ from \eqref{eqPsigma}.

\begin{prop}\label{prophm2}
We have
\[
\mathpzc{Ber}^\bs(T^\bs(u)e^{-\pa_u})=\Big(\prod_{i=1}^N \mathfrak D_i^\bs(u+\wp_i^\bs)^{\ka_i}\Big)e^{-(\mathfrak m-\mathfrak n)\pa_u}.
\]
In particular, 
\[
\mathfrak P_\sigma^\bs: \prod_{i=1}^N \mathfrak D_i^\bs(u+\wp_i^\bs)^{\ka_i}\mapsto \prod_{i=1}^N \mathfrak D_i^{\tl\bs}(u+\wp_i^{\tl\bs})^{\tl \ka_i}.
\]
\end{prop}

Recall $\mathpzc{Ber}^\bs(u)$ from \eqref{eq:gammadef}--\eqref{cu}. Then we have the following. (Note that $\mathpzc{Ber}(u)$ is the Berezinian of the matrix $X(u)e^{-\pa_u}$ defined in this subsection.)

\begin{thm}\label{main4}
Let $\bs,\tl\bs$ be two symmetric parity sequences in $S_{\mathfrak m|\mathfrak n}$. If we identify the superalgebras $\Y^\bs$ and $\Y^{\tl\bs}$ via the isomorphism \eqref{eqPsigma2}, then we have $\mathpzc{Ber}^\bs(u)=\mathpzc{Ber}^{\tl\bs}(u)$.
\end{thm}
\begin{proof}
Note that if $N=2\ell+1$, then the parity $\ka_{\ell}$ is fixed and depends on the parity of $\mathfrak m$ and $\mathfrak n$. Note that $\mathfrak m$ and $\mathfrak n$ cannot be both odd. In addition, a symmetric parity sequence $\bs$ is determined by its parity subsequence $\bs_{[1,\ell]}$. Let $\tl\bs$ be another symmetric parity sequence in $S_{\fkm|\fkn}$. Then a permutation $\sigma\in\mathfrak S_{N}$ with $\ka_{i}=\tl{\ka}_{\sigma(i)}$ and $\sigma(i)'=\sigma(i')$ for $i\in\I$ can be chosen to satisfy the property that $\{1,\cdots,\ell\}$ is invariant under $\sigma$.

Moreover, if $N=2\ell+1$, then, by the quasi-determinant presentation of $d_i(u)$ in \eqref{gd1}, the series $d_{\ell+1}(u)$ remains the same under the permutation $\sigma$ we choose. Thus, by \eqref{B=CC}, it suffices to show that $\mathfrak C^\bs(u)$ is independent of $\bs$.

Recall the definition of $\rY(\gl_{[\ell]}^{\bm\ka})$ from \eqref{ysbs-sub}. By \eqref{bcom}, there is a homomorphism from
\[
\rY(\gl_{[\ell]}^{\bm\ka})\to \Y^\bs,\quad t_{ij}(u)\mapsto x_{ij}(u),\quad 1\lle i,j\lle \ell,
\]
see also Proposition \ref{secA} below. Moreover, the homomorphism sends $\mathfrak D_i(u)$ to $d_i(u)$ for $1\lle i\lle \ell$. Recall from \eqref{eq:gammadef} that
$\gamma_1=\varkappa-\frac{1}{2}\ka_1$.
Since $\varkappa$ remains the same for all symmetric parity sequences, the claims follow from Proposition \ref{prophm2}.
\end{proof}

\section{Relations between Gaussian generators}\label{sec:lowrk}
In this section, we work out the relations of type A and the relations between Gaussian generators when $N=2,3,4,5$. For low rank situation, there will be two main cases, i.e. with a fixed point ($N=2,4$) for the Dynkin diagram automorphism $\tau$ and without a fixed point ($N=3,5$).
\subsection{Relations of type A}\label{secA}
Suppose $N\gge 2m \gge  4$. Recall the definition of $\rY(\gl_{[m]}^{\bm\ka})$ from \eqref{ysbs-sub}. By \eqref{bcom}, there is a homomorphism from
\[
\rY(\gl_{[m]}^{\bm\ka})\to \Y^\bs,\quad t_{ij}(u)\mapsto x_{ij}(u),\quad 1\lle i,j\lle m.
\]
Therefore, the relations among $d_i(u),e_{j}(v),f_{k}(w)$ for $1\lle i\lle m$ and $1\lle j,k<m$ are the same as  those in $\rY(\gl_{[m]}^{\bm\ka})$. These relations are given in \cite{Gow2007gauss,Peng2016parabolic,tsymbaliuk2020shuffle} which we shall list below. 

Let $\ell=\lfloor\tfrac{N}2\rfloor$. For $1\lle i<\ell$, set
\[
e_i^\circ(u)=\sum_{r\gge 2}e_i^{(r)}u^{-r},\qquad f_i^\circ(u)=\sum_{r\gge 2}f_i^{(r)}u^{-r},\qquad \zeta_i(u)=\tl d_i(u)d_{i+1}(u).
\]

\begin{prop}\label{propA}
The following relations hold in $\scrX^\bs$, with the conditions on the indices $1\lle i,j<\ell$ and $1\lle k,l\lle \ell$,
\begin{align*}
&[d_k(u),d_l(v)]=0,\\
&[e_i(u),f_j(v)]=\delta_{ij}\ka_{i+1}\frac{\zeta_i(u)-\zeta_i(v)}{u-v},\\
&[d_k(u),e_j(v)]=\ka_k(\delta_{k,j+1}-\delta_{kj})\frac{d_k(u)(e_j(u)-e_j(v))}{u-v},\\
&[d_k(u),f_j(v)]=\ka_k(\delta_{kj}-\delta_{k,j+1})\frac{(f_j(u)-f_j(v))d_k(u)}{u-v},\\
&[e_i(u),e_i(v)]=\ka_i\frac{(e_i(u)-e_i(v))^2}{u-v},\\
&[f_i(u),f_i(v)]=-\ka_i\frac{(f_i(u)-f_i(v))^2}{u-v},\\
&[e_i(u),e_j(v)]=[f_i(u),f_j(v)]=0,\qquad\quad~~  \text{ if }c_{ij}=0.
\end{align*}
Moreover, we have, for $1\lle i\lle \ell-2$,
\begin{align*}
u[e_i^\circ(u),e_{i+1}(v)]-v[e_i(u),e_{i+1}^\circ(v)]&=\ka_{i+1}e_i(u)e_{i+1}(v),\\
u[f_{i+1}(v),f_i^\circ(u)]-v[f_{i+1}^\circ(v),f_i(u)]&=\ka_{i+1}f_{i+1}(v)f_{i}(u),
\end{align*}
and the cubic Serre relations, for $1\lle i,j<\ell$ with $|i-j|=1$,
\begin{align*}
\big[e_i(u),[e_i(v),e_j(w)]\big]+\big[e_i(v),[e_i(u),e_j(w)]\big]&=0,\\
\big[f_i(u),[f_i(v),f_j(w)]\big]+\big[f_i(v),[f_i(u),f_j(w)]\big]&=0,
\end{align*}
and the quartic Serre relations, for $1\lle i<\ell-1$, $|\alpha_i|=\bar 1$, $|\alpha_{i-1}|=|\alpha_{i+1}|=\bar 0$,
\begin{align*}
\big[[e_{i-1}(u),e_i^{(1)}],[e_i^{(1)},e_{i+1}(v)]\big]=\big[[f_{i-1}(u),f_i^{(1)}],[f_i^{(1)},f_{i+1}(v)]\big]=0.
\end{align*}
\end{prop}

Comparing to \cite{lu2024drinfeld}, the quartic Serre relation is new. We verify the relation \eqref{qsconjnew} here. By Proposition \ref{propnewserre} (and its proof), it suffices to establish the following.

\begin{lem}\label{newserrelem}
If $i+1\lle \lfloor \frac{N}2\rfloor$ or $i-1\gge \lfloor \frac{N+1}2\rfloor$, then 
$$\big[[e_{i-1}^{(1)},e_i^{(1)}],[e_i^{(1)},e_{i+1}^{(1)}]\big]=\big[[f_{i-1}^{(1)},f_i^{(1)}],[f_i^{(1)},f_{i+1}^{(1)}]\big]=0.$$
\end{lem}
\begin{proof}
We first consider the case $i+1\lle \lfloor \frac{N}2\rfloor$ and $\big[[e_{i-1}^{(1)},e_i^{(1)}],[e_i^{(1)},e_{i+1}^{(1)}]\big]$. By \eqref{bcom}, we have
\begin{align*}
[x_{i-1,i}^{(1)},x_{i,i+1}^{(1)}]=\ka_ix_{i-1,i+1}^{(1)},\quad [x_{i,i+1}^{(1)},x_{i+1,i+2}^{(1)}]=\ka_{i+1}x_{i,i+2}^{(1)}.
\end{align*}
Again by \eqref{bcom}, we have $[x_{i-1,i+1}^{(1)},x_{i,i+2}^{(1)}]=0$. Since $e_j^{(1)}=x_{j,j+1}^{(1)}$, the desired relation follows.

The other situations are obtained by applying the anti-automorphism $\eta$ and Lemma \ref{efaiii} to $$\big[[e_{i-1}^{(1)},e_i^{(1)}],[e_i^{(1)},e_{i+1}^{(1)}]\big]=0;$$ see also Lemma \ref{lem:eta}.
\end{proof}

Recall Lemma \ref{efaiii} and the definition of $h_i(u),b_i(u)$ from \eqref{beven}--\eqref{hodd}. It is clear that the commutator relations \eqref{qsconj0}--\eqref{qsconjnew} between the generating series    $h_i(u),b_i(u)$ for $i$ not close to $\ell$ can be deduced from the relations listed in Proposition \ref{propA} (exactly as $\rY(\mathfrak{sl}^\bs_{[\ell]})$) and Lemma \ref{newserrelem}; see also Lemma \ref{lemalt} and Lemma \ref{lemnew1}. Thus, we are left with verifying the relations for $i$ close to $\ell$. By Lemma \ref{shiftlem}, it suffices to do that for the case when $N$ is small, namely $N=2,3,4,5$. Note that the case $N=4,5$ is mainly for the Serre relations.

\subsection{Relations in the case $N=2$}
In this case, we have $\ka_1=\ka_2$. All elements here are even.
\begin{lem}\label{b2ser}
We have the following relations in $\scrX^{(\ka_1,\ka_2)}$,
\begin{align}
[d_i(u),d_j(v)]&=0,\label{b2dd}\\
e(u)&=-f(-u),\label{b2e=f}\\
\tl d_1(u)d_2(u)&=\tl d_1(-u)d_2(-u),\label{b2d=d}\\
[d_1(u),f(v)]&=\frac{\ka_1}{u-v}(f(u)-f(v))d_1(u)+\frac{\ka_1}{u+v}d_1(u)(e(u)+f(v)),\label{b2d1f}\\
[d_2(u),f(v)]&=\frac{\ka_1}{u-v}(f(v)-f(u))d_2(u)-\frac{\ka_1}{u+v}d_2(u)(e(u)+f(v)),\label{b2d2f}\\
[f(u),f(v)]&=-\frac{\ka_1}{u-v}(f(u)-f(v))^2+\frac{\ka_1}{u+v}\big(\tl d_1(u)d_2(u)-\tl d_1(v)d_2(v)\big).\label{b2ff}
\end{align}
\end{lem}
\begin{proof}
Equations \eqref{b2dd}--\eqref{b2d=d} follow directly from Lemma \ref{efaiii} and Lemma \ref{lemnew1}.

\mybox{Equations \eqref{b2d1f}--\eqref{b2d2f}}. Setting $i=j=k=1$ and $l=2$ in \eqref{bcom} and using \eqref{b2dd}, we have
\beq\label{b2pf1}
(u^2-v^2)[d_1(u),e(v)]=\ka_1(u+v)d_1(u)(e(v)-e(u))-\ka_1(u-v)(e(v)+f(u))d_1(u).
\eeq
Thus \eqref{b2d1f} follows from \eqref{b2e=f} and \eqref{b2pf1}. By Lemma \ref{efaiii}, we have $d_2(u)=c(u)\tl d_1(-u)$. Note that $c(u)$ is central. Therefore \eqref{b2d2f} follow from \eqref{b2e=f} and \eqref{b2d1f}.

\mybox{Equation \eqref{b2ff}}. Setting $i=k=1$ and $j=l=2$, we have
\[
(u+v)[x_{12}(u),x_{12}(v)]=\ka_1\big([x_{12}(u),x_{12}(v)]+x_{11}(u)x_{22}(v)-x_{11}(v)x_{22}(u)\big).
\]
Rewriting it in terms of Gaussian generating series, we obtain
\beq\label{b2pf3}
\begin{split}
&\, d_1(u)e(u)d_1(v)e(v)-d_1(v)e(v)d_1(u)e(u)\\
=&\,\frac{\ka_1}{u+v}\Big(d_1(u)e(u)d_1(v)e(v)-d_1(v)e(v)d_1(u)e(u)+d_1(u)d_2(v)\\&\qquad \qquad +d_1(u)f(v)d_1(v)e(v)-d_1(v)d_2(u)-d_1(v)f(u)d_1(u)e(u)\Big)
\end{split}
\eeq
By \eqref{b2pf1}, we have
\beq\label{b2pf4}
e(v)d_1(u)=d_1(u)e(v)-\frac{\ka_1}{u-v}d_1(u)(e(v)-e(u))+\frac{\ka_1}{u+v}(e(v)+f(u))d_1(u).
\eeq
Using \eqref{b2pf4} to commute $e(u)d_1(v)$ and $e(v)d_1(u)$ in \eqref{b2pf3}, we find that the l.h.s. of \eqref{b2pf3} is transformed to
\begin{align*}
d_1(u)d_1(v)[e(u),e(v)]&-\frac{\ka_1}{u-v}d_1(u)d_1(v)(e(u)-e(v))^2\\ &+\frac{\ka_1}{u+v}\Big(d_1(u)e(u)d_1(v)e(v)+d_1(u)f(v)d_1(v)e(v)\\&\qquad\qquad-d_1(v)e(v)d_1(u)e(u)-d_1(v)f(u)d_1(u)e(u)\Big).
\end{align*}
Thus it follows from \eqref{b2pf3} that
\[
[e(u),e(v)]=\frac{\ka_1}{u-v}(e(u)-e(v))^2-\frac{\ka_1}{u+v}\big(\tl d_1(u)d_2(u)-\tl d_1(v)d_2(v)\big).
\]
By \eqref{b2e=f}, we obtain \eqref{b2ff}.
\end{proof}

It is convenient to use the following notation. We write
\[
A(u,v)\simeq B(u,v)
\]
if $A(u,v)$ and $B(u,v)$ have the same coefficients of $u^{-r-1}v^{-s-1}$ for $r,s\in\bN$. We sometimes use the same notation for the case $r\in\Z$ and $s\in\bN$. When the case $r\in\bZ$ is used, we will clarify it further.

Recall $b(u)=\sqrt{-1}f(u)$ and $h(u)=\tl d_1(u)d_2(u)$ from \eqref{beven} and \eqref{heven}, respectively. Here we drop the subscript $i$ as the rank is one. 

\begin{lem}\label{b2s}
We have the relations in $\scrX^{(\ka_1,\ka_2)}$ in terms of generating series,
\begin{align*}
&[h(u),h(v)]=0,\qquad h(u)=h(-u),\\
&[b(u),b(v)]=-\frac{\ka_1}{u-v}(b(u)-b(v))^2-\frac{\ka_1}{u+v}(h(u)-h(v)),\\\
&[h(u),b(v)]\simeq\frac{1}{u^2-v^2}\big((2\ka_1v+1)h(u)b(v)+(2\ka_1v-1)b(v)h(u)\big).
\end{align*}
\end{lem}
\begin{proof}
The proof is parallel to that of \cite[Lem. 4.2]{Lu2023drinfeld}.
\end{proof}

\begin{prop}\label{propx=2}
In terms of components, we have
\begin{align*}
[h_r,h_s] &=0,\\
[b_{r+1},b_s]-[b_r,b_{s+1}] &=\ka_1\big(b_rb_s+b_sb_r\big)-2(-1)^r\ka_1h_{r+s+1},\\
[h_{r+2},b_s]-[h_r,b_{s+2}] &=2\ka_1(b_{s+1}h_r+h_rb_{s+1})+[h_r,b_s].
\end{align*}
Here in the last equality we allow $r\in\bZ$ with $h_{-1}=1$ and $h_{r}=0$ for $r<-1$.
\end{prop}
\begin{proof}
The proof is parallel to that of  \cite[Prop. 4.3]{Lu2023drinfeld} by taking the coefficients of $u^{-r-1}v^{-s-1}$ for $r,s\in\bN$ from the relations (expanded in the region $|u|\gg |v|$) in Lemma \ref{b2s}. Note that we can take $r\in\bZ$ in the third relation as the power of $v$ in the terms we dropped are nonnegative; see the
proof of \cite[Prop. 7.16]{lu2024drinfeld} for more detail.
\end{proof}
\begin{rem}
If we set $b(u)=f(u)$, then we have
\[
[b(u),b(v)]=-\frac{\ka_1}{u-v}(b(u)-b(v))^2 + \frac{\ka_1}{u+v}(h(u)-h(v)).
\]
The purpose to use $b(u)=\sqrt{-1}f(u)$ is to change $+$ above to $-$ so that it will match with the nonsuper split case in \cite{Lu2023drinfeld}.
\end{rem}

\begin{lem}\label{centralb2}
The coefficients of $d_1(u)d_2(u-\ka_1)$ are central elements in $\scrX^\bs$.
\end{lem}
\begin{proof}
Since $[d_i(u),d_j(v)]=0$ and $e(u)=-f(-u)$ by \eqref{b2dd} and \eqref{b2e=f}, respectively, it suffices to prove that $[d_1(u)d_2(u-\ka_1),f(v)]=0$. By ignoring the terms like $f(u)d_1(u)$ and $d_1(u)e(u)$ in \eqref{b2d1f} and \eqref{b2d2f} (as these series do not contribute if we expand them in the region $|u|\gg |v|$), we have
\begin{align*}
d_1(u)f(v)\simeq \frac{(u+v)(u-v-\ka_1)}{(u+v-\ka_1)(u-v)}f(v)d_1(u),\\
d_2(u)f(v)\simeq \frac{(u+v)(u-v+\ka_1)}{(u+v+\ka_1)(u-v)}f(v)d_2(u).
\end{align*}
Thus $d_1(u)d_2(u-\ka_1)f(v)\simeq f(v)d_1(u)d_2(u-\ka_1)$ and the statement follows.
\end{proof}

\subsection{Relations in the case $N=4$}\label{secN=4}
Let us consider the case $N=4$ with $\bm\ka=(\ka_1,\ka_2,\ka_3,\ka_4)$ such that $\ka_1=\ka_4$ and $\ka_2=\ka_3$. Thanks to Lemma \ref{shiftlem}, we have the rank-reduction homomorphism
\[
\psi_1^\bs:\scrX^{(\ka_2,\ka_3)}\to \scrX^\bs,\quad d_{i}(u)\to d_{i+1}(u),\quad e(u)\mapsto e_{2}(u),\quad f(u)\mapsto f_{2}(u).
\]
Thus, by Lemma \ref{efaiii}, Lemma \ref{b2ser} and Proposition \ref{propA}, we immediately have the following relations
\begin{align}
&[d_i(u),d_j(v)]=[d_1(u),e_2(v)]=[d_1(u),f_2(v)]=0,\label{ddb4}\\
&(u-v)[e_1(u),f_1(v)]=\ka_2\big(\tl d_1(u)d_2(u)-\tl d_1(v)d_2(v)\big),\label{e1f1b4}\\
&(u-v)[e_1(u),e_1(v)]=\ka_1(e_1(u)-e_1(v))^2,\label{e1e1b4}\\
&(u-v)[d_1(u),e_1(v)]=\ka_1d_1(u)(e_1(v)-e_1(u)),\label{d1e1b4}\\
&(u-v)[d_2(u),e_1(v)]=\ka_2d_2(u)(e_1(u)-e_1(v)),\label{d2e1b4}\\
&(u-v)[d_1(u),f_1(v)]=\ka_1(f_1(u)-f_1(v))d_1(u),\label{d1f1b4}\\
&[e_2(u),e_2(v)]=\frac{\ka_2}{u-v}\big(e_2(u)-e_2(v)\big)^2-\frac{\ka_2}{u+v}\big(\tl d_2(u)d_3(u)-\tl d_2(v)d_3(v)\big).\label{e2e2b4}
\end{align}

We reformulate \eqref{e1e1b4} when $|\alpha_1|=\bar 1$ is odd for later use.
\begin{lem}\label{odd0}
If $\ka_1\ne \ka_{2}$, then the relation \eqref{e1e1b4} is equivalent to $[e_1(u),e_1(v)]=0$.
\end{lem}
\begin{proof}
It is straightforward since if $\ka_1\ne \ka_{2}$, then the LHS of \eqref{e1e1b4} is anti-symmetric in $u,v$ while the RHS of \eqref{e1e1b4} is symmetric in $u,v$.
\end{proof}

\begin{lem}\label{be1e2lem}
We have
\begin{align}
(u-v)[e_1(u),e_2(v)]=\ka_2\big(e_1(u)e_2(v)-e_1(v)e_2(v)-e_{13}(u)+e_{13}(v)\big),\label{be1e2}\\
(u-v)[f_2(v),f_1(u)]=\ka_2\big(f_2(v)f_1(u)-f_2(v)f_1(v)-f_{31}(u)+f_{31}(v)\big).\label{bf1f2}
\end{align}
\end{lem}
\begin{proof}
By Lemma \ref{efaiii} or applying the anti-automorphism $\eta$ from Lemma \ref{lem:eta}, it suffices to show \eqref{be1e2}. By \eqref{bcom}, we have 
\[
(u-v)[x_{12}(u),x_{23}(v)]=\ka_2\big(x_{22}(u)x_{13}(v)-x_{22}(v)x_{13}(u)\big).
\]
Note that $|\alpha_2|=\bar 0$. In terms of Gaussian generators, we have
\beq\label{b4pf1}
\begin{split}
(u-v)&[d_1(u)e_1(u),d_2(v)e_2(v)+f_1(v)d_1(v)e_{13}(v)]\\=\, &\ka_2\big(d_2(u)+f_1(u)d_1(u)e_1(u)\big)d_1(v)e_{13}(v)-\ka_2\big(d_2(v)+f_1(v)d_1(v)e_1(v)\big)d_1(u)e_{13}(u).
\end{split}
\eeq
We shall transform the LHS of \eqref{b4pf1}. Expanding the commutator, the LHS of \eqref{b4pf1} becomes
\begin{align*}
(u-v)\Big(d_1(u)& e_1(u)d_2(v) e_2(v)-d_2(v) e_2(v)d_1(u) e_1(u)\\
+&\, d_1(u) e_1(u)f_1(v) d_1(v)e_{13}(v)-f_1(v)d_1(v)e_{13}(v)d_1(u)e_1(u)\Big).
\end{align*}
Permuting the products $e_1(u)d_2(v)$, $e_2(v)d_1(u)$, and $e_1(u)f_1(v)$ using \eqref{d2e1b4}, \eqref{ddb4}, \eqref{e1f1b4}, respectively, we have
\begin{align*}
&\,d_1(u) \big((u-v)d_2(v)e_1(u)-\ka_2d_2(v)(e_1(u)-e_1(v))\big) e_2(v)-(u-v)d_2(v) d_1(u)e_2(v) e_1(u)\\
+&\,d_1(u) \big((u-v)\ka_1\ka_2f_1(v)e_1(u)+\ka_2\tl d_1(u)d_2(u)-\ka_2\tl d_1(v)d_2(v)\big) d_1(v)e_{13}(v)\\
&\hskip8.2cm-(u-v)f_1(v)d_1(v)e_{13}(v)d_1(u)e_1(u),
\end{align*}
which simplifies further to
\begin{align*}
d_1(u)d_2(v)\big((u-v)[e_1(u),e_2(v)]-\ka_2(e_1(u)-e_1(v))e_2(v)-\ka_2e_{13}(v)\big)+\ka_2d_1(v)d_2(u)e_{13}(v)\\
+(u-v)\big(\ka_1\ka_2d_1(u)f_1(v)e_1(u)d_1(v)e_{13}(v)-f_1(v)d_1(v)e_{13}(v)d_1(u)e_1(u)\big).
\end{align*}
Therefore, it suffices to show that
\begin{align*}
(u-v)\big(&\,\ka_1\ka_2d_1(u)f_1(v)e_1(u)d_1(v)e_{13}(v)-f_1(v)d_1(v)e_{13}(v)d_1(u)e_1(u)\big)\\
=&\,\ka_2\big(f_1(u)d_1(u)e_1(u) d_1(v)e_{13}(v)- f_1(v)d_1(v)e_1(v) d_1(u)e_{13}(u)\big)
\end{align*}
Applying \eqref{d1f1b4}, i.e. $(u-v)d_1(u)f_1(v)-\ka_1f_1(u)d_1(u)=(u-v)f_1(v)d_1(u)-\ka_1f_1(v)d_1(u)$, it reduces to show
\beq\label{b4pf2}
(u-v)[d_1(u)e_1(u),d_1(v)e_{13}(v)]=\ka_1\big(d_1(u)e_1(u)d_1(v)e_{13}(v)-d_1(v)e_1(v)d_1(u)e_{13}(u)\big),
\eeq
which is equivalent to $(u-v)[x_{12}(u),x_{13}(v)]=\ka_1\big(x_{12}(u)x_{13}(v)-x_{12}(v)x_{13}(u)\big)$ and follows from \eqref{bcom}.
\end{proof}

\begin{lem}\label{lemhelp1}
If $\ka_1=\ka_2$ \emph{(}i.e. $|\alpha_1|=\bar 0$\emph{)}, then we have
\beq\label{e1e13}
[e_1(u),e_{13}(v)-e_1(v)e_2(v)]=-[e_1(u),e_2(v)]e_1(u).
\eeq
\end{lem}
\begin{proof}
It follows from \eqref{bcom} that
\[
(u-v)[x_{11}(u),x_{13}(v)]=\ka_1\big(x_{11}(u)x_{13}(v)-x_{11}(v)x_{13}(u)\big),
\]
which implies that
\beq\label{b4pf3}
(u-v)[d_1(u),e_{13}(v)]=\ka_1d_1(u)(e_{13}(v)-e_{13}(u)).
\eeq
We have
\beq\label{b4pf4}
\begin{split}
\ka_1\big(d_1(u)e_1(u)d_1(v)e_{13}(v)-&\, d_1(v)e_1(v)d_1(u)e_{13}(u)\big)\\
\stackrel{\eqref{d1e1b4}}{=}&\, d_1(u)\big(\ka_1d_1(v)e_1(u)+\tfrac{1}{u-v}d_1(v)(e_1(u)-e_1(v))\big)e_{13}(v)\\
-&\,d_1(v)\big(\ka_1d_1(u)e_1(v)+\tfrac{1}{u-v}d_1(u)(e_1(u)-e_1(v))\big)e_{13}(u).
\end{split}
\eeq
On the other hand, we also have
\beq\label{b4pf5}
\begin{split}
\ka_1\big(d_1(u)e_1(u)& d_1(v)e_{13}(v)- d_1(v)e_1(v)d_1(u)e_{13}(u)\big)\\
&\, \stackrel{\eqref{b4pf2}}{=}(u-v)\big(d_1(u)e_1(u)d_1(v)e_{13}(v)-d_1(v)e_{13}(v)d_1(u)e_{1}(u)\big)\\
&\, \xlongequal{\eqref{d1e1b4}\eqref{b4pf3}}d_1(u)\big((u-v)d_1(v)e_1(u)+\ka_1d_1(v)(e_1(u)-e_1(v))\big)e_{13}(v)\\
& \qquad\qquad \, -d_1(v)\big((u-v)d_1(u)e_{13}(v)-\ka_1d_1(u)(e_{13}(v)-e_{13}(u))\big)e_1(u).
\end{split}
\eeq
Combining \eqref{b4pf4} and \eqref{b4pf5}, we obtain
\beq\label{b4pf6}
\begin{split}
(u-v)[e_1(u),e_{13}(v)]=&\,\frac{1}{u-v}\big(e_1(u)-e_1(v)\big)\big(e_{13}(v)-e_{13}(u)\big)\\&+\ka_1e_1(v)\big(e_{13}(v)-e_{13}(u)\big)-\ka_1\big(e_{13}(v)-e_{13}(u)\big)e_1(u).
\end{split}
\eeq
To prove \eqref{e1e13}, it suffices to show
\begin{align*}
(u-v)&\,[e_1(u),e_{13}(v)]\\
=&\,(u-v)\big([e_1(u),e_1(v)]e_2(v)+e_1(v)[e_1(u),e_2(v)]-[e_1(u),e_2(v)]e_1(u)\big)\\
\xlongequal{\eqref{e1e1b4}\eqref{be1e2}}&\, \ka_1e_1(u)[e_1(u),e_2(v)]-\ka_1[e_1(u),e_1(v)e_2(v)]\\
&+\ka_1e_1(v)\big(e_{13}(v)-e_{13}(u)\big)-\ka_1\big(e_{13}(v)-e_{13}(u)\big)e_1(u).
\end{align*}
Therefore, by \eqref{b4pf6}, it reduces to show
\beq\label{b4pf7}
\ka_1e_1(u)[e_1(u),e_2(v)]-\ka_1[e_1(u),e_1(v)e_2(v)]=\frac{1}{u-v}\big(e_1(u)-e_1(v)\big)\big(e_{13}(v)-e_{13}(u)\big).
\eeq
Rewrite $(u-v)\big(e_1(u)[e_1(u),e_2(v)]-[e_1(u),e_1(v)e_2(v)]\big)$ as
\[
(u-v)\big(e_1(u)[e_1(u),e_2(v)]-[e_1(u),e_1(v)]e_2(v)-e_1(v)[e_1(u),e_2(v)]\big),
\]
then \eqref{b4pf7} follows from \eqref{e1e1b4} and \eqref{be1e2}.
\end{proof}

\begin{lem}\label{sec:A-serre}
We have the following Serre relation,
\[
\big[e_1(u),[e_1(v),e_2(w)]\big]+\big[e_1(v),[e_1(u),e_2(w)]\big]=0.
\]
\end{lem}
\begin{proof}
If $\ka_1=\ka_2$ (i.e. $|\alpha_1|=\bar 0$), then the proof is parallel to that of \cite[Lemma 6.3]{Peng2016parabolic} as the commutator relations used in the proofs are the same, see \eqref{e1e1b4}, \eqref{be1e2}, \eqref{e1e13}. If $\ka_1\ne \ka_2$ (i.e. $|\alpha_1|=\bar 1$), then the  relation follows from \eqref{e1e1b4}.
\end{proof}

\begin{rem}
Note that the calculations above involve only relations of nontwisted super Yangians of type A; cf. the proof of Corollary \ref{cor:commu}.
\end{rem}

\begin{lem}\label{x4serre0}
We have the following finite type Serre relation,
\[
\big[e_2^{(1)},[e_2^{(1)},e_1^{(1)}]\big]=e_1^{(1)}.
\]
\end{lem}
\begin{proof}
Note that $|e_2^{(1)}|=\bar 0$. By \eqref{bcom}, we have
\[
(u-v)[x_{13}(u),x_{32}(v)]=\ka_2\big(x_{33}(u)x_{12}(v)-x_{33}(v)x_{12}(u)\big).
\]
Thus $[x_{13}(u),x_{32}^{(1)}]=\ka_2 x_{12}(u)$ which transforms to
\[
[d_1(u)e_{13}(u),f_2^{(1)}]=\ka_2d_1(u)e_1(u).
\]
By \eqref{ddb4}, we conclude that $[e_{13}(u),f_2^{(1)}]=\ka_2e_1(u)$. Finally, the desired relation follows from Lemma \ref{efaiii} as $e_2(u)=-f_2(-u)$ implies that $f_2^{(1)}=e_2^{(1)}$.
\end{proof}

Recall $b_{i,r}$ and $h_{i,r}$ for $i\in\{1,2,3\}$,  $r\in\bN$ from \eqref{beven}, \eqref{heven}, and \eqref{bhcom}.

\begin{prop}
The relations \eqref{qsconj0}--\eqref{qsconj8} hold in $\scrX^\bs$.
\end{prop}
\begin{proof}
The proof is similar to the proof of the nonsuper case \cite[Prop. 7.16]{lu2024drinfeld} by using the relations established in this subsection; see also Proposition \ref{X=3comp} below.
\end{proof}

\subsection{Relations in the case $N=3$}
In this case, we have $\bm\ka=(\ka_1,\ka_2,\ka_3)$ such that $\ka_1=\ka_3$. We start with listing relations between Gaussian generators $d_1(u)$, $d_2(u)$, $e_1(u)$, $f_1(u)$, cf. Lemma \ref{efaiii}.
\begin{lem}
We have 
\begin{align}
[d_i(u),d_j(v)]&=0,\label{ddb3}\\
[d_1(u),e_1(v)]&=\frac{\ka_1}{u-v}d_1(u)(e_1(v)-e_1(u)),\label{d1e1b3}\\
[d_1(u),f_1(v)]&=\frac{\ka_1}{u-v}(f_1(u)-f_1(v))d_1(u),\label{d1f1b3}\\
[e_1(u),e_1(v)]&=\frac{\ka_1}{u-v}(e_1(u)-e_1(v))^2,\label{e1e1b3}\\
[f_1(u),f_1(v)]&=-\frac{\ka_1}{u-v}(f_1(u)-f_1(v))^2,\label{f1f1b3}\\
[e_1(u),f_1(v)]&=\frac{\ka_2}{u-v}\big(\tl d_1(u)d_2(u)-\tl d_1(v)d_2(v)\big)+\frac{\ka_2}{u+v}\big(e_{13}(u)+e_1(u)f_1(v)+f_{31}(v)\big),\label{e1f1b3}\\
[d_2(u),e_1(v)]&=\frac{\ka_2}{u-v}d_2(u)(e_1(u)-e_1(v))-\frac{\ka_2}{u+v}(e_1(v)+f_2(u))d_2(u),\label{d2e1b3}\\
[d_2(u),f_1(v)]&=\frac{\ka_2}{u-v}(f_1(v)-f_1(u))d_2(u)+\frac{\ka_2}{u+v}d_2(u)(f_1(v)+e_2(u)).\label{d2f1b3}
\end{align}
\end{lem}
\begin{proof}
We verify the essential relations as the other relations follow from the essential ones by taking the anti-automorphism $\eta$; see Lemma \ref{lem:eta}.

The relation \eqref{ddb3} follows from Lemma \ref{lemnew1}. Then \eqref{d1e1b3} follows from \eqref{bcom} with $i=j=k=1$, $l=2$ and $[d_1(u),d_1(v)]=0$. Applying the anti-automorphism $\eta$ to \eqref{d1e1b3}, we obtain \eqref{d1f1b3}.

\mybox{Equation \eqref{e1f1b3}}. We first claim that
\beq\label{b3pf1}
\begin{split}
(u-v)\big([d_1(u)e_1(u),f_1(v)d_1(v)]&-d_1(u)[e_1(u),f_1(v)]d_1(v)\big)\\
&=\ka_2\big(f_1(u)d_1(u)e_1(u)d_1(v)-f_1(v)d_1(v)e_1(v)d_1(u)\big).
\end{split}
\eeq
This is equivalent to
\begin{align*}
\big((u-v)d_1(u)f_1(v)&-\ka_1f_1(u)d_1(u)\big)e_1(u)d_1(v)\\
&=f_1(v)d_1(v)\big((u-v)d_1(u)e_1(u)-\ka_1e_1(v)d_1(u)\big).
\end{align*}
Note that by \eqref{d1f1b3}, we have
\[
(u-v)d_1(u)f_1(v)-\ka_1f_1(u)d_1(u)=f_1(v)\big((u-v)d_1(u)-\ka_1d_1(u)\big).
\]
Hence, to prove \eqref{b3pf1}, it reduces to show
\[
\big((u-v)d_1(u)-\ka_1d_1(u)\big)e_1(u)d_1(v)=d_1(v)\big((u-v)d_1(u)e_1(u)-\ka_1e_1(v)d_1(u)\big),
\]          
that is
\[
(u-v)d_1(u)[d_1(v),e_1(u)]=\ka_1\big(d_1(v)e_1(v)d_1(u)-d_1(u)e_1(u)d_1(v)\big).
\]
Applying \eqref{d1e1b3} to the LHS, it transforms to
\[
d_1(v)[d_1(u),e_1(v)]=d_1(u)[d_1(v),e_1(u)]
\]
which follows directly from \eqref{d1e1b3} by applying it to both sides.

Let us come back to \eqref{e1f1b3}. By \eqref{bcom} with $i=l=1$ and $j=k=2$ in terms of Gaussian generators, we have
\begin{align*}
(u^2-v^2)&[d_1(u)e_1(u),f_1(v)d_1(v)]\\
=\, &(u+v)\ka_2\big(d_2(u)d_1(v)+f_1(u)d_1(u)e_1(u)d_1(v)-d_2(v)d_1(u)-f_1(v)d_1(v)e_1(v)d_1(u)\big)\\
&\qquad \qquad +(u-v)\ka_2\big(d_1(u)e_{13}(u)d_1(v)+d_1(u)e_1(u)f_1(v)d_1(v)+d_1(u)f_{31}(v)d_1(v)\big).
\end{align*}
Now using \eqref{b3pf1} for $[d_1(u)e_1(u),f_1(v)d_1(v)]$ and multiplying $\tl d_1(u)$, $\tl d_1(v)$ from the left and the right, respectively, one finds \eqref{e1f1b3}. 

\mybox{Equation \eqref{e1e1b3}}. By \eqref{bcom} with $i=k=1$ and $j=l=2$, we have $$[x_{12}(u),x_{12}(v)]=\ka_1(x_{12}(u)x_{12}(v)-x_{12}(v)x_{12}(u)).$$
There are two cases depending on the parity of $|\alpha_1|$. 

(1) If $\ka_1=\ka_2$, then we have $[x_{12}(u),x_{12}(v)]=0$ which implies that
\be 
d_1(u) e_1(u)d_1(v) e_1(v)=d_1(v)  e_1(v)d_1(u) e_1(u).
\ee 
Using \eqref{d1e1b3} to commute $d_1(u)$ and $e_1(v)$, we have
\begin{align*}
d_1(u)\Big(d_1(v)e_1(u)+&\frac{\ka_1}{u-v}d_1(v)(e_1(u)-e_1(v))\Big)e_1(v)\\
&=d_1(v)\Big(d_1(u)e_1(v)+\frac{\ka_1}{u-v}d_1(u)(e_1(u)-e_1(v))\Big)e_1(u).
\end{align*}
Canceling $d_1(u)d_1(v)$, we obtain \eqref{e1e1b3}.

(2) If $\ka_1\ne \ka_2$, then 
\beq\label{b3pf2}
[x_{12}^{(1)},x_{12}(v)]=0\Longrightarrow [e_1^{(1)},d_1(v)e_{1}(v)]=0.
\eeq
Therefore, we have
\beq\label{995}
d_1(v)[e_1^{(1)},e_1(v)]=[d_1(v),e_1^{(1)}]e_1(v).
\eeq
It follows from \eqref{d1e1b3} that 
\begin{align}
&[d_1^{(1)},e_1(v)]=\ka_1 e_1(v),\label{992}\\
&[d_1(u),e_1^{(1)}]=\ka_1d_1(u)e_1(u),\label{993}\\
&[d_1^{(2)},e_1(v)]-v[d_1^{(1)},e_1(v)]=\ka_1 d_1^{(1)}e_1(v)-\ka_1 e_1^{(1)}.\label{994}
\end{align}
Combining \eqref{995} and \eqref{993}, we have
\beq\label{996}
[e_1^{(1)},e_1(v)]=\ka_1e_1(v)^2.
\eeq
Set $\mathfrak{d}_{1,1}=d_1^{(2)}-\frac12 (d_{1}^{(1)})^2-\frac{1}{2}\ka_1d_1^{(1)}$, then it follows from \eqref{992} and \eqref{994} that
\beq\label{997}
[\mathfrak{d}_{1,1},e_1(v)]=\ka_1(ve_1(v)-e_1^{(1)})~\Longrightarrow ~[\mathfrak{d}_{1,1},e_1^{(r)}]=\ka_1e_1^{(r+1)}.
\eeq

By Lemma \ref{odd0}, it suffices to show that $[e_1(u),e_1(v)]=0$, i.e. $[e_1^{(r)},e_1^{(s)}]=0$ for all $r,s\in\bN$. Considering the coefficient of $v^{-1}$ in \eqref{996}, it is immediate that $[e_1^{(1)},e_1^{(1)}]=0$ or $(e_1^{(1)})^2=0$. Similarly,
\begin{align}
[e_1^{(1)},e_1^{(2)}]=\ka_1 (e_1^{(1)})^2=0,\quad [e_1^{(1)},e_1^{(3)}]=\ka_1 [e_1^{(1)},e_1^{(2)}]=0.\label{998}
\end{align}
Applying $[\mathfrak{o}_{1,1},\cdot\,]$ to the first equality of \eqref{998}, we obtain
\[
0=\big[\mathfrak{o}_{1,1},[e_1^{(1)},e_1^{(2)}]\big]=\ka_1[e_1^{(2)},e_1^{(2)}]+\ka_1[e_1^{(1)},e_1^{(3)}].
\]
Hence, by \eqref{998}, we have $[e_1^{(2)},e_1^{(2)}]=0$. Now we prove by induction on $r+s$ that $[e_1^{(r)},e_1^{(s)}]=0$ whose base case is proved above. By induction hypothesis and \eqref{996}, we have
\beq\label{999}
[e_1^{(1)},e_1^{(r+s)}]=\frac{\ka_1}{2}\sum_{i=1}^{r+s-1}[e_1^{(i)},e_1^{(r+s-i)}]=0.
\eeq
Applying $[\mathfrak{o}_{1,1},\cdot\,]$ to $[e_1^{(r)},e_1^{(s)}]=0$, we have
\beq\label{9910}
[e_1^{(r+1)},e_1^{(s)}]+[e_1^{(r)},e_1^{(s+1)}]=0.
\eeq
The desired statement follows from \eqref{999} and \eqref{9910}.

\mybox{Equation \eqref{d2e1b3}}. Taking the coefficients of $u$ in \eqref{bcom} with $i=1,j=k=l=2$ in terms of Gaussian generators, we find that
\beq\label{b3pf3}
[e_1^{(1)},d_2(v)+f_1(v)d_1(v)e_1(v)]=\ka_2\big(d_1(v)e_1(v)+f_2(v)d_2(v)+f_{31}(v)d_1(v)e_1(v)\big).
\eeq
It follows from \eqref{b3pf2} that 
\beq\label{b3pf4}
[e_1^{(1)},d_1(v)e_1(v)]=0.
\eeq
Note also that by \eqref{e1f1b3} we have
\beq\label{b3pf5}
[e_1^{(1)},f_1(v)]=\ka_2\big(1-\tl d_1(v)d_2(v)+f_{31}(v)\big).
\eeq
Combining \eqref{b3pf3}, \eqref{b3pf4}, and \eqref{b3pf5}, we conclude that
\beq\label{b3pf6}
[e_1^{(1)},d_2(v)]=\ka_2\big(d_2(v)e_1(v)+f_2(v)d_2(v)\big).
\eeq
By Lemma \ref{efaiii}, we have $f_2(v)=-e_1(-v)$. Therefore, we find that
\begin{align*}
d_1(u)[e_1(u),d_2(v)]&\stackrel{\eqref{ddb3}}{=}[d_1(u)e_1(u),d_2(v)]\stackrel{\eqref{993}}{=}\ka_1[d_1(u),[e_1^{(1)},d_2(v)]]\\
&\stackrel{\eqref{b3pf6}}{=}\ka_1\ka_2[d_1(u),d_2(v)e_1(v)+f_2(v)d_2(v)]\\
&\stackrel{\eqref{ddb3}}{=}\ka_1\ka_2d_2(v)[d_1(u),e_1(v)]+\ka_1\ka_2[d_1(u),f_2(v)]d_2(v)\\
&\stackrel{\eqref{d1e1b3}}{=}\frac{\ka_2}{u-v}d_2(v)d_1(u)\big(e_1(v)-e_1(u)\big)+\frac{\ka_2}{u+v}d_1(u)\big(e_1(u)+f_2(v)\big)d_2(v),
\end{align*}
completing the proof of \eqref{d2e1b3}.
\end{proof}

\begin{lem}\label{centralb3}
The coefficients of $d_1(u+\frac{\ka_1+\ka_2}{2})^{\ka_1}d_2(u)^{\ka_2}d_3(u-\frac{\ka_2+\ka_3}{2})^{\ka_3}$ are central elements in $\scrX^\bs$.
\end{lem}
\begin{proof}
The proof is similar to that of Lemma \ref{centralb2}.
\end{proof}

Recall $b_{i,r}$ and $h_{i,r}$ for $i\in\{1,2\}$,  $r\in\bN$ from \eqref{bodd}, \eqref{hodd}, and \eqref{bhcom}.
\begin{prop}\label{X=3comp}
We have 
\begin{align}
&[h_{i,r},h_{j,s}]=0,\qquad h_{i,r}=(-1)^{r+1}h_{\tau i,r},\label{hhb3}\\
&[b_{i,r+1},b_{j,s}]  - [b_{i,r },b_{j,s+1 }]  =\frac{c_{ij}}{2}\{b_{i,r },b_{j,s}\}-2\delta_{\tau i, j} (-1)^{r}\ka_ih_{j,r+s+1 },\label{bbb3}\\
&[h_{i,r+2},b_{j,s}] - [h_{i,r},b_{j,s+2}]=\frac{c_{ij}-c_{i\tau j}}{2}\{h_{i,r+1},b_{j,s}\}\notag\\
&\qquad\qquad\qquad \qquad \qquad~~ \ \  +\frac{c_{ij}+c_{i\tau j}}{2} \{h_{i,r},b_{j,s+1}\}+\frac{c_{ij}c_{\tau i,j}}{4} [h_{i,r}, b_{j,s}].\label{hbb3}
\end{align}
Here in the last equality we allow $r\in\bZ$ with $h_{i,-1}=1$ and $h_{i,r}=0$ for $r<-1$.
\end{prop}
\begin{proof}
We give a brief proof; see the proof of \cite[Prop. 7.20]{lu2024drinfeld} for more detail.

The relation \eqref{hhb3} is clear from \eqref{i=tauih}, Lemma \ref{lemnew1}, and the definition of $h_{i,r}$. The relation \eqref{bbb3} for $i=j$ is obvious from \eqref{e1e1b3} and \eqref{f1f1b3}. Then we consider the case $(i,j)=(1,2)$. It follows from \eqref{e1f1b3} and $e_1(u)=-f_2(-u)$ (by Lemma \ref{efaiii}) that
\begin{align*}
(u-v)[f_1(u),f_2(v)]&\simeq -\frac{\ka_1(u-v)}{u+v}\tl d_2(v)d_3(v)-\ka_1f_2(v)f_1(u),
\end{align*}
in the region $|u|\gg |v|$. Here we dropped terms like $\tl d_1(u)d_2(u)$, $f_{31}(u)$, $e_{13}(-v)$ and used $\tl d_2(v)d_3(v)=\tl d_1(-v)d_2(-v)$ from Lemma \ref{efaiii}. Thus we have
\[
\Big(u-v+\frac{\ka_2}{2}\Big)[b_1(u),b_2(v)]\simeq\Big(\ka_1-\frac{2\ka_1v}{u+v}\Big)h_2(v)-\ka_1b_2(v)b_1(u),
\]
which implies further
\[
(u-v)[b_1(u),b_2(v)]\simeq -\frac{\ka_2}{2}\big\{b_1(u),b_2(v)\big\}-\frac{2\ka_1v}{u+v} h_2(v).
\]
By taking the coefficients of $u^{-r-1}v^{-s-1}$ with $r,s\in\bN$, we obtain the relation \eqref{bbb3} for $(i,j)=(1,2)$.

We consider the relation \eqref{hbb3}. Note that $h_1(u)=h_2(-u)$ by \eqref{i=tauih} and the anti-automorphism $\eta$ sends $b_1(u)$ to a scalar multiple of $b_2(-u)$. It suffices to consider the case $(i,j)=(1,1)$. It follows from \eqref{d1f1b3} and \eqref{d2f1b3} that
\beq\label{1helper}
\begin{split}
[\tl d_1(u)&d_2(u),f_1(v)]\\&=\frac{\ka_1+\ka_2}{u-v}\tl d_1(u)(f_1(v)-f_1(u))d_2(u)+\frac{\ka_2}{u+v}\tl d_1(u)d_2(u)(f_1(v)+e_2(u)).    
\end{split}
\eeq
There are two cases depending on the parity of $\alpha_1$. 

(1) If $\ka_1\ne \ka_2$, then $c_{11}=c_{22}=\ka_1+\ka_2=0$ and $c_{12}=c_{21}=-\ka_2$. Thus by \eqref{1helper} we have
\[
(u+v)[h_1(u),b_1(v)]\simeq\frac{\ka_2}{2}\{h_1(u),b_1(v)\}.
\]
Multiplying it by $(u-v)$ and taking the coefficients of $u^{-r-1}v^{-s-1}$ with $r\in\Z$ and $s\in\bN$, one finds the relation \eqref{hbb3} for $(i,j)=(1,1)$ for this case.

(2) If $\ka_1=\ka_2$, then by \eqref{1helper} we have
\beq\label{2hper}
[\tl d_1(u)d_2(u),f_1(v)]\simeq\frac{\ka_1+\ka_2}{u-v}\tl d_1(u)f_1(v)d_2(u)+\frac{\ka_2}{u+v}\tl d_1(u)d_2(u)f_1(v),
\eeq
where we dropped terms like $\tl d_1(u)f_1(u)d_2(u)$. By \eqref{d1f1b3}, we find
\[
\frac{1}{u-v}\tl d_1(u)f_1(v)\simeq \frac{1}{u-v-\ka_1}f_1(v)\tl d_1(u).
\]
Plugging it into \eqref{2hper}, we obtain
\[
[\tl d_1(u)d_2(u),f_1(v)]\simeq\frac{\ka_1+\ka_2}{u-v-\ka_1}f_1(v)\tl d_1(u)d_2(u)+\frac{\ka_2}{u+v}\tl d_1(u)d_2(u)f_1(v),
\]
Substituting $u\mapsto u+\frac{1}4\ka_2, v\mapsto v+\frac{1}{4}\ka_2$ and clearing the denominator, we obtain from \eqref{bodd}--\eqref{hodd} that
\begin{align*}\Big(u+v+\frac{\ka_2}{2}\Big)&\big(u-v-\ka_2\big)\big[h_1(u),b_1(v)\big]\\
\simeq\,&2\ka_2\Big(u+v+\frac{\ka_2}{2}\Big)b_1(v)h_1(u)+\ka_2\big(u-v-\ka_2\big)h_1(u)b_1(v),
\end{align*}
which gives rise to
\begin{align*}
\big(u^2-v^2\big)\big[h_1(u),b_1(v)\big]
\simeq \Big(\frac{3\ka_2}{2}u+\frac{\ka_2}{2}v\Big)\big\{h_1(u),b_1(v)\big\}-\frac{1}{2}[h_1(u),b_1(v)].                               
\end{align*}
Taking the coefficients of $u^{-r-1}v^{-s-1}$ with $r\gge -2,s\in\bN$, one finds the relation \eqref{hbb3} for $(i,j)=(1,1)$ if $\ka_1=\ka_2$.
\end{proof}

By Proposition \ref{prop-serre}, we need one more lemma for the Serre relations among generators of degree zero, which will be
sufficient to deduce more general Serre relation \eqref{qsconj8}.
\begin{lem}\label{fin-serre}
We have
\[
\big[b_{1,0},[b_{1,0},b_{2,0}]\big]=2(1+\ka_1\ka_2)b_{1,0},\quad \big[b_{2,0},[b_{2,0},b_{1,0}]\big]=2(1+\ka_1\ka_2)b_{2,0}.
\]
\end{lem}
\begin{proof}
We only prove the first relation. The other one is similar or can be obtained by taking the anti-automorphism $\eta$ from Lemma \ref{lem:eta}.

It follows from  \eqref{bodd}, \eqref{e1f1b3}, and Lemma \ref{efaiii} that
\[
[b_{1,0},b_{2,0}]=-[f_1^{(1)},f_2^{(1)}]=-[f_1^{(1)},e_1^{(1)}]=-\ka_1\Big(h_{1,0}-\frac{1}4-f_{31}^{(1)}\Big).
\]
On the other hand, setting $i=2,j=k=3,l=1$ in \eqref{bcom}, we find
\[
[x_{23}(u),x_{31}(v)]=\frac{\ka_1}{u-v}\big(x_{33}(u)x_{21}(v)-x_{33}(v)x_{21}(u)\big),
\]
which implies further
\begin{align*}
[b_{1,0},f_{31}^{(1)}]&=\sqrt{-1}\,[f_{1}^{(1)},f_{31}^{(1)}]=\sqrt{-1}\,[e_{2}^{(1)},f_{31}^{(1)}]\\&=\sqrt{-1}\,[x_{23}^{(1)},x_{31}^{(1)}]=\sqrt{-1}\,\ka_1x_{21}^{(1)} =\sqrt{-1}\,\ka_1f_1^{(1)}=\ka_1b_{1,0}.
\end{align*}
Since by \eqref{hbb3} that $[h_{1,0},b_{1,0}]=(\ka_1+2\ka_2)b_{1,0}$, the first equality follows.
\end{proof}

\subsection{Relations in the case $N=5$}To complete, we still need to verify certain Serre relations which reduce to the case of $\scrX^\bs$. This is very similar to Section \ref{secN=4}. It turns out that the calculations of this case will be exactly the same as the Yangians $\rY(\gl^\bs_{[m]})$.  

\begin{lem}\label{mmm}
We have the following Serre relations,
\begin{align*}
\big[e_1(u),[e_1(v),e_2(w)]\big]+\big[e_1(v),[e_1(u),e_2(w)]\big]=0,\\
\big[e_2(u),[e_2(v),e_1(w)]\big]+\big[e_2(v),[e_2(u),e_1(w)]\big]=0.
\end{align*}
\end{lem}
\begin{proof}
The lemma above can be proved similarly to Corollary \ref{cor:commu} as follows. The series $e_1(u)$ is expressed in terms of $x_{11}(u)$ and $x_{12}(u)$ while $e_2(u)$ is expressed in terms of $x_{11}(u)$, $x_{13}(u)$, $x_{21}(u)$ and $x_{23}(u)$. Note that the commutator relations between these series $x_{ab}(u)$ are the same as in $\rY(\gl^\bs_{[3]})$ (with $x_{ab}(u)$ replaced by $t_{ab}(u)$); see \eqref{Trel} and \eqref{bcom}. Thus these Serre relations hold as the same Serre relations hold in $\rY(\gl^\bs_{[3]})$.
\end{proof}
\begin{rem}
Alternatively, one can prove the lemma following the strategy of \cite[Lem. 7.24]{lu2024drinfeld}.
\end{rem}

\begin{cor}
We have the following Serre relations in $\scrX^\bs$,
\begin{align*}
\big[b_i(u),[b_i(v),b_j(w)]\big]+\big[b_i(v),[b_i(u),b_j(w)]\big]=0,
\end{align*}
for the pairs $(i,j)=(1,2),(2,1), (3,4),(4,3)$.
\end{cor}
\begin{proof}
Follows from Lemma \ref{mmm} by applying the anti-automorphism $\eta$ and Lemma \ref{efaiii}.
\end{proof}

The Serre relations for the pairs $(i,j)=(2,3),(3,2)$ follow from the corresponding relations in $\scrX^\bs_{[2,2']}$ with the rank-reduction homomorphism $\psi_1^\bs$; see Lemma \ref{shiftlem}, Lemma \ref{fin-serre}, and Proposition \ref{prop:Serre}.

\bibliographystyle{amsalpha}
\bibliography{all}

\end{document}